\documentclass[11pt,a4paper]{article}

\usepackage{epsf,epsfig,amsfonts,amsgen,amsmath,amstext,amsbsy,amsopn,amsthm,amsfonts,amssymb,amscd
}
\usepackage{ebezier,eepic}
\usepackage{color}
\usepackage{multirow}
\newtheorem{thm}{Theorem}[section]

\newtheorem{lem}[thm]{Lemma}
\newtheorem{false statement}{False statement}
\newtheorem{cor}[thm]{Corollary}
\newtheorem{fact}{Fact}

\theoremstyle{definition}

\newtheorem{claim}{Claim}

\newtheorem{conj}{Conjecture}

\newtheorem{case}{Case}

\newcounter{mathitem}
  {\begin{list}{{$(\roman{mathitem})$}}{
   \setcounter{mathitem}{0}
   \usecounter{mathitem}
   \setlength{\topsep}{0pt plus 2pt minus 0pt}
   \setlength{\parskip}{0pt plus 2pt minus 0pt}
   \setlength{\partopsep}{0pt plus 2pt minus 0pt}
   \setlength{\parsep}{0pt plus 2pt minus 0pt}
   \setlength{\leftmargin}{35pt}
   \setlength{\itemsep}{0pt plus 2pt minus 0pt}}}
  {\end{list}}

\begin{document}

\title{\bf\Large Stability results on the circumference of a graph}

\date{}

\author{
Jie Ma\footnote{School of Mathematical Sciences, University of Science
and Technology of China, Hefei, Anhui 230026, China. Email: jiema@ustc.edu.cn.
Research supported in part by National Natural Science Foundation of China grants 11501539 and 11622110 and Anhui Initiative in Quantum information Technologies grant AHY150200.}~~~~~~~
Bo Ning\footnote{Center for Applied Mathematics,
Tianjin University, Tianjin 300072, China. Email: bo.ning@tju.edu.cn. Research supported in part by National Natural Science Foundation of China grant 11601379.}
}
\maketitle

\begin{abstract}

In this paper, we extend and refine previous Tur\'an-type results on graphs with a given circumference.
Let $W_{n,k,c}$ be the graph obtained from a clique $K_{c-k+1}$
by adding $n-(c-k+1)$ isolated vertices each joined to the same $k$ vertices of the clique,
and let $f(n,k,c)=e(W_{n,k,c})$.
Improving a celebrated theorem of Erd\H{o}s and Gallai \cite{EG59},
Kopylov \cite{K77} proved that for $c<n$, any 2-connected graph $G$ on $n$ vertices with circumference $c$ has at most
$\max\{f(n,2,c),f(n,\lfloor\frac{c}{2}\rfloor,c)\}$ edges,
with equality if and only if $G$ is isomorphic to $W_{n,2,c}$ or $W_{n,\lfloor\frac{c}{2}\rfloor,c}$.
Recently, F\"{u}redi et al. \cite{FKV16,FKLV-arxiv} proved a stability version of Kopylov's theorem.
Their main result states that if $G$ is a 2-connected graph on $n$ vertices with circumference $c$
such that $10\leq c<n$ and $e(G)>\max\{f(n,3,c),f(n,\lfloor\frac{c}{2}\rfloor-1,c)\}$, then either $G$ is a subgraph of $W_{n,2,c}$ or $W_{n,\lfloor\frac{c}{2}\rfloor,c}$,
or $c$ is odd and $G$ is a subgraph of a member of two well-characterized families which we define as $\mathcal{X}_{n,c}$ and $\mathcal{Y}_{n,c}$.

We prove that if $G$ is a 2-connected graph on $n$ vertices with minimum degree at least $k$ and circumference $c$
such that $10\leq c<n$ and $e(G)>\max\{f(n,k+1,c),f(n,\lfloor\frac{c}{2}\rfloor-1,c)\}$,
then one of the following holds:\\
(i) $G$ is a subgraph of $W_{n,k,c}$ or $W_{n,\lfloor\frac{c}{2}\rfloor,c}$, \\
(ii) $k=2$, $c$ is odd, and $G$ is a subgraph of a member of $\mathcal{X}_{n,c}\cup \mathcal{Y}_{n,c}$, or \\
(iii) $k\geq 3$ and $G$ is a subgraph of the union of a clique $K_{c-k+1}$ and some cliques $K_{k+1}$'s, where any two cliques share the same two vertices.

This provides a unified generalization of the above result of F\"{u}redi et al. \cite{FKV16,FKLV-arxiv} as well as
a recent result of Li et al. \cite{LN16} and independently, of F\"{u}redi et al. \cite{FKL17} on non-Hamiltonian graphs.
A refinement and some variants of this result are also obtained.
Moreover, we prove a stability result on a classical theorem of Bondy \cite{B71} on the circumference.
We use a novel approach, which combines several proof ideas including a closure operation and an edge-switching technique.
We will also discuss some potential applications of this approach for future research.
\end{abstract}



\section{Introduction}

All graphs in this paper are simple and finite.
The {\it circumference} $c(G)$ of a graph $G$ is the length of a longest cycle in $G$.
A graph $G$ is called {\it Hamiltonian} if $c(G)=|V(G)|$.
Let $\delta(G)$ and $e(G)$ denote the minimum degree and the number of edges in $G$, respectively.

Determining the circumference of a graph is a classical problem in graph theory.
It is well known that even determining if the graph is Hamiltonian is NP-hard.
There has been extensive research investigating various relations between the circumference and other natural graph parameters.
One such example is the famous theorem proved by Dirac \cite{D52} in 1952, which states that
for any 2-connected graph $G$, $c(G)\geq \min \{2\delta(G), |V(G)|\}$.
In this paper, we mainly focus on the Tur\'an-type problems on the circumference.
One cornerstone in this direction is the following celebrated Erd\H{o}s-Gallai theorem.

\begin{thm}[Erd\H{o}s and Gallai \cite{EG59}]\label{Th:EG-cyc}
For any graph $G$ on $n$ vertices, $e(G)\leq\frac{c(G)(n-1)}{2}$.\footnote{For a graph $G$ without cycles, we view $c(G)=2$.}
\end{thm}

\noindent This is sharp if $n-1$ is divisible by $c-1$ (where $c:=c(G)$),
by considering the graph consisting of cliques $K_{c}$'s sharing
only one common vertex.
Theorem \ref{Th:EG-cyc} also implies that
if an $n$-vertex graph $G$ contains no paths of length $k$,\footnote{We specify that throughout
this paper, a path of length $k$ has $k$ edges (and hence $k+1$ vertices).}
then $e(G)\leq \frac{(k-1)n}{2}$.

Bondy \cite{B71} generalized this theorem by showing the following.

\begin{thm}[Bondy \cite{B71}]\label{Le:Bondy}
Let $G$ be a graph on $n$ vertices and let $C$ be a longest cycle of $G$
of length $c$. Then the number of edges with at most one endpoint in $C$ is at most $\frac{c}{2}\cdot (n-c)$.
In addition, if $G$ is 2-connected, then this number is at most $\lfloor\frac{c}{2}\rfloor\cdot (n-c)$.
\end{thm}

\noindent Since there are at most $\binom{c}{2}$ edges spanned in $V(C)$,
we see that Theorem \ref{Le:Bondy} indeed is a strengthening of Theorem \ref{Th:EG-cyc}.
\footnote{An improved version for 2-connected graphs can be found in Fan \cite{F90}.}

Throughout this paper, let $W_{n,k,c}$ be the graph obtained from a clique $K_{c-k+1}$
by adding $n-(c-k+1)$ isolated vertices each joined to the same $k$ vertices of $K_{c-k+1}$, and
\begin{align*}
f(n,k,c):=\binom{c-k+1}{2}+k\cdot (n-c+k-1).
\end{align*}
So $W_{n,k,c}$ has $n$ vertices, minimum degree $k$ and circumference $c$ with $e(W_{n,k,c})=f(n,k,c)$.

\subsection{Stability on non-Hamiltonian graphs with large minimum degree}
For non-Hamiltonian graphs $G$ (that is, $c(G)\leq n-1$),
Ore \cite{O61} proved that $e(G)\leq \binom{n-1}{2}+1=f(n,1,n-1)$.
This was generalized further by Erd\H{o}s \cite{E62}.

\begin{thm}[Erd\H{o}s \cite{E62}]\label{ThEr}
If $G$ is a non-Hamiltonian graph on $n$ vertices with $\delta(G)\geq k$, where $1\leq
k\leq(n-1)/2$, then $e(G)\leq \max\{f(n,k,n-1),f(n,\lfloor\frac{n-1}{2}\rfloor,n-1)\}$.
\end{thm}

This bound is sharp for all $1\leq k\leq(n-1)/2$.
Recently, Li and Ning \cite{LN16}, and independently,
F\"{u}redi, Kostochka and Luo \cite{FKL17} proved a stability version of this theorem.

\begin{thm}[\cite{LN16,FKL17}]\label{Th:nonH}
Let $G$ be a non-Hamiltonian graph on $n$ vertices with $\delta(G)\geq k$,
where $1\leq k\leq(n-1)/2$.
If $e(G)>\max\{f(n,k+1,n-1),f(n,\lfloor\frac{n-1}{2}\rfloor,n-1)\}$,
then $G$ is a subgraph of either $W_{n,k,n-1}$ or the edge-disjoint union of two cliques $K_{n-k}$ and $K_{k+1}$ sharing a common vertex.
\end{thm}

Very recently, F\"{u}redi, Kostochka and Luo obtained a stronger stability theorem (and also some other related results) in \cite{FKL-arxiv}.

\subsection{Stability on graphs with given circumference}
There are many refinements of Theorem \ref{Th:EG-cyc} in the literature,
see \cite{FS75, L75, W76, K77} or the survey \cite{FS}.
Among them, Kopylov \cite{K77} proved the following strong version in 1977.

\begin{thm}[Kopylov \cite{K77}]\label{Th:Kopylov}
Let $G$ be a 2-connected graph on $n$ vertices.
If $c(G)=c\leq n-1$, then
$e(G)\leq \max\{f(n,2,c),f(n,\lfloor\frac{c}{2}\rfloor,c)\}$.
\end{thm}

We also mention that another proof of Theorem \ref{Th:Kopylov} was found by Fan, Lv, and Wang \cite{FLW04} in 2004.
Using an edge-switching technique, the authors of \cite{FLW04} proved a slightly stronger result when $n-1\geq c(G)\geq \frac{2n}{3}+1$.
This, together with a result of Woodall \cite{W76} that if $G$ is a 2-connected graph with circumference $c\leq \frac{2n+2}{3}$
then $e(G)\leq f(n,\lfloor\frac{c}{2}\rfloor,c)$, gives a different proof of Theorem \ref{Th:Kopylov}.
More importantly for us, the technique of \cite{FLW04} provides an integral ingredient to the proof of our main theorem
(see Subsection \ref{subsec:5.2}).

In 2016, F\"{u}redi, Kostochka, and Verstra\"{e}te \cite{FKV16} proved a stability result of Theorem \ref{Th:Kopylov} in the range of $n\geq 3\lfloor \frac{c}{2}\rfloor$.
Together with this, F\"{u}redi, Kostochka, Luo, and Verstra\"{e}te \cite{FKLV-arxiv}
recently obtained a completed stability version of the above theorem of Kopylov.
To state their result, we need to introduce two families $\mathcal{X}_{n,c}$ and $\mathcal{Y}_{n,c}$,
which contain graphs of a given circumference $c$ where $c$ is odd, as follows:

-- A graph $G$ in the family $\mathcal{X}_{n,c}$ has $n$ vertices and $V(G)=A\cup B\cup X$
such that
$G[A]$ induces a clique $K_{\lfloor\frac{c}{2}\rfloor}$,
both $G[B]$ and $G[X]$ are stable,
$(A,B)$ is complete bipartite, and
there exist two vertices $a\in A$ and $b\in B$ such that for any $x\in X$, $N_G(x)=\{a,b\}$.

-- A graph $G$ in the family $\mathcal{Y}_{n,c}$ has $n$ vertices and $V(G)=A\cup B\cup Y$
such that
$G[A]$ induces a clique $K_{\lfloor\frac{c}{2}\rfloor}$,
$G[B]$ is stable,
$G[Y]$ is a nontrivial star forest\footnote{We say a star forest is {\it nontrivial}, if it has at least two stars and every star has at least one edge.},
$(A,B)$ is complete bipartite, and
there exist two vertices $a, b\in A$ such that every star $S$ in $G[Y]$ is {\it $\{a,b\}$-feasible}:
that is, $N_G(S)=\{a,b\}$ and if $|V(S)|\geq 3$, then all leaves of $S$ have degree 2 in $G$ and have a common neighbor in $\{a,b\}$.

\begin{thm}[F\"{u}redi, Kostochka, Luo, and Verstra\"{e}te \cite{FKLV-arxiv}]\label{Th:FKLV}
Let $G$ be a 2-connected graph on $n$ vertices with circumference $c$, where $10\leq c\leq n-1$.
If $e(G)>\max\{f(n,3,c),f(n,\lfloor\frac{c}{2}\rfloor-1,c)\}$,
then one of the following conclusions holds:\\
(a) $G\subseteq W_{n,2,c}$,\\
(b) $G\subseteq W_{n,\lfloor\frac{c}{2}\rfloor,c}$, or\\
(c) if $c$ is odd, then $G$ is a subgraph of a member of $\mathcal{X}_{n,c}\cup \mathcal{Y}_{n,c}$.
\end{thm}

\noindent We remark that the case $c\leq 9$ was also fully characterized in \cite{FKV16,FKLV-arxiv};
in particular, the case $c=9$ requires another extremal graph, besides those stated in Theorem \ref{Th:FKLV}.
As a corollary in \cite{FKLV-arxiv}, if in addition $G$ is 3-connected in Theorem \ref{Th:FKLV},
then one must have $G\subseteq W_{n,\lfloor\frac{c}{2}\rfloor,c}$.

By imposing minimum degree as a new parameter, Woodall \cite{W76}
asked the following refinement of Theorem \ref{Th:EG-cyc} in 1976.

\begin{conj}[Woodall \cite{W76}]
Let $G$ be a 2-connected graph on $n$ vertices with $\delta(G)\geq k$.
If $c(G)=c\leq n-1$, then $e(G)\leq \max\{f(n,k,c),f(n,\lfloor\frac{c}{2}\rfloor,c)\}$.
\end{conj}

\noindent One may also view this as a unification of Theorems \ref{ThEr} and \ref{Th:Kopylov}.
It should be mentioned that Kopylov's original proof in \cite{K77} can be modified to give a solution of this conjecture.

The Tur\'an-type problem of cycles of given lengths for graphs with a given minimum degree is well-studied
(see Chapter 5 of \cite{Bol} for an inclusive discussion).

\subsection{The main result}
Our main result is a stability version of Woodall's conjecture,
which also is a unified generalization of Theorem \ref{Th:FKLV} and Theorem \ref{Th:nonH} for 2-connected graphs.
We define the graph $Z_{n,k,c}$ to be the union of a clique $K_{c-k+1}$ and $\frac{n-(c-k+1)}{k-1}$ cliques $K_{k+1}$'s
such that any two cliques share the same two vertices.

\begin{thm}\label{Th:main}
Let $G$ be a 2-connected graph on $n$ vertices with $\delta(G)\geq k$
and circumference $c$, where $10\leq c\leq n-1$.\footnote{Following the proofs, we shall see that the same statement also holds for the case $c=8$.}
If
$e(G)>\max\left\{f\left(n,k+1,c\right),f\left(n,\left\lfloor\frac{c}{2}\right\rfloor-1,c\right)\right\},
$
then one of the following conclusions holds:\\
(a) $G\subseteq W_{n,k,c}$,\\
(b) $G\subseteq W_{n,\lfloor\frac{c}{2}\rfloor,c}$, \\
(c) if $k=2$ and $c$ is odd, then $G$ is a subgraph of a member of $\mathcal{X}_{n,c}\cup \mathcal{Y}_{n,c}$, or\\
(d) if $k\geq 3$, then $G\subseteq Z_{n,k,c}$.
\end{thm}

We make some remarks.
First, we see that the case $k=2$ of Theorem \ref{Th:main} gives the precise statement of Theorem \ref{Th:FKLV}.
Secondly, by letting $c=n-1$, Theorem \ref{Th:main} also provides a refined version of Theorem \ref{Th:nonH} for 2-connected graphs.
Also we have $c\geq 2k$ in Theorem \ref{Th:main},
which follows by Dirac's theorem that $c\geq \min\{n,2k\}$.
Note that $Z_{n,k,c}$ has $n$ vertices, minimum degree $k$ (assuming $c\geq 2k$) and circumference $c$
with
$$e(Z_{n,k,c})=\binom{c-k+1}{2}+\frac{k+2}{2}\cdot(n-c+k-1).$$
Thus in certain range it holds that $e(Z_{n,k,c})> \max\left\{f\left(n,k+1,c\right),f\left(n,\left\lfloor\frac{c}{2}\right\rfloor-1,c\right)\right\}$.

We also notice that every graph in $\mathcal{X}_{n,c}\cup \mathcal{Y}_{n,c}$ has a vertex of degree 2,
and the graph $Z_{n,k,c}$ has a 2-cut.
Therefore, it is prompt to deduce that

\begin{cor}
Let $G$ be a 3-connected graph on $n$ vertices with $\delta(G)\geq k$
and circumference $c$, where $10\leq c\leq n-1$.
If
$e(G)>\max\left\{f\left(n,k+1,c\right),f\left(n,\left\lfloor\frac{c}{2}\right\rfloor-1,c\right)\right\},
$
then either $G\subseteq W_{n,k,c}$ or $G\subseteq W_{n,\lfloor\frac{c}{2}\rfloor,c}$.
\end{cor}



\subsection{A refinement}
Using a novel closure operation which we define below,
we are able to refine Theorem \ref{Th:main} in more detail.
We point out that the closure operation has proved to be a powerful tool for finding long cycles (see \cite{BC76,BM08,R97}).
However it is surprising for us that in some cases one can even precisely describe the extremal graphs using closures.

The {\it \bf $k$-closure} of a graph $G$ is the
graph obtained from $G$ by recursively joining pairs of nonadjacent vertices whose
degree sum is at least $k$ until no such pair remains.
We also say that the resulting graph is {\it \bf $k$-closed}.
Let $G$ be a graph and $C$ be a cycle of $G$ of length $c$.
The \emph{\bf $C$-closure} of $G$, denoted as $\overline{G}$, is obtained from $G$
by replacing the subgraph $G[C]$ by its $(c+1)$-closure.
It is crucial to observe that $G\subseteq \overline{G}$.

\begin{thm}\label{Th:main-refined}
Let $G$ be a 2-connected graph on $n$ vertices with $\delta(G)\geq k$
and let $C$ be a longest cycle in $G$ of length $c\in [10,n-1]$.
If
$
e(G)>\max\left\{f\left(n,k+1,c\right),f\left(n,\left\lfloor\frac{c}{2}\right\rfloor-1,c\right)\right\},
$
then one of the following holds:\\
(a) $\overline{G}=W_{n,k,c}$, where $\overline{G}$ denotes the $C$-closure of $G$,\\
(b) $G\subseteq W_{n,\lfloor\frac{c}{2}\rfloor,c}$, \\
(c) if $k=2$ and $c$ is odd, then $G$ is a subgraph of a member of $\mathcal{X}_{n,c}\cup \mathcal{Y}_{n,c}$, or\\
(d) if $k\geq 3$, then $\overline{G}=Z_{n,k,c}$.
\end{thm}

\subsection{Two variants}
The following two variants of the main result also can be obtained analogously,
from which we see how the extremal graphs of Theorem \ref{Th:main-refined} change
as the parameters vary in the function $f(n,k,c)$.

\begin{thm}\label{Th:Var1}
Let $G$ be a 2-connected graph on $n$ vertices with $\delta(G)\geq k$
and let $C$ be a longest cycle in $G$ of length $c\in [10,n-1]$.
If
$
e(G)>\max\left\{f\left(n,k+1,c\right),f\left(n,\left\lfloor\frac{c}{2}\right\rfloor,c\right)\right\},
$
then $\overline{G}=W_{n,k,c}$ or $Z_{n,k,c}$, where $\overline{G}$ denotes the $C$-closure of $G$.
\end{thm}

\begin{thm}\label{Th:Var2}
Let $G$ be a 2-connected graph on $n$ vertices with $\delta(G)\geq k$
and let $C$ be a longest cycle in $G$ of length $c\in [10,n-1]$.
If
$
e(G)>\max\left\{f\left(n,k,c\right),f\left(n,\left\lfloor\frac{c}{2}\right\rfloor-1,c\right)\right\},
$
then either $G\subseteq W_{n,\lfloor\frac{c}{2}\rfloor,c}$, or $k=2$, $c$ is odd and $G$ is a subgraph of a member of $\mathcal{X}_{n,c}\cup \mathcal{Y}_{n,c}$.
\end{thm}

In particular, if we choose $c=n-1$ in Theorem \ref{Th:Var1},
then it follows that $\overline{G}=W_{n,k,n-1}$.
This is because that $Z_{n,k,n-1}$ is valid only for $k=2$, but when $k=2$, $W_{n,2,c}$ and $Z_{n,2,c}$ are identical.
This provides another refined version of Theorem \ref{Th:nonH} for 2-connected graphs.

\subsection{Stability on a theorem of Bondy}

Our other result on the circumference of a graph is a stability version of Theorem \ref{Le:Bondy}.

\begin{thm}\label{Th:BondyStab-2con}
Let $G$ be a 2-connected graph on $n$ vertices and $C$ be a longest cycle in $G$ of length $c$, where $10\leq c\leq n-1$.
If the number of edges with at most one endpoint in $C$ is more than $\left(\left\lfloor\frac{c}{2}\right\rfloor-1\right)(n-c),$
then either $G\subseteq W_{n,\lfloor\frac{c}{2}\rfloor,c}$,
or $c$ is odd and $G$ is a subgraph of a member of $\mathcal{X}_{n,c}\cup \mathcal{Y}_{n,c}$.
\end{thm}

\subsection{Proof reduction}
In this subsection, we give a sketch of the proof of Theorem \ref{Th:main-refined},
which we emphasize is quite different from the existing ones in \cite{FKL17,FKLV-arxiv,FKV16}.

The proof will be split into two parts, according to the simple observation that
given a longest cycle $C$ in the graph $G$ which has many edges,
either the number of edges with at most one endpoint in $C$ is large or the number of edges spanned in $V(C)$ is large.
The former case will be dealt with by Theorem \ref{Th:BondyStab-2con},
and the latter case will be handled by the following result.

Define $h(n,k):=\binom{n-k}{2}+k(k-1)$. We point out that $h(n+1,k)=e(W_{n,k,n}).$

\begin{thm}\label{Th:WZ+closure1}
Let $G$ be a 2-connected graph on $n$ vertices with $\delta(G)\geq k$ and
$C$ be a longest cycle in $G$ of length $c\in [6,n-1]$.
If $e(G)>\max\left\{f\left(n,k+1,c\right),f\left(n,\left\lfloor\frac{c}{2}\right\rfloor-1,c\right)\right\}$
and
$e(G[C])>h\left(c+1,\left\lfloor\frac{c}{2}\right\rfloor-1\right),$
then either $G\subseteq W_{n,\lfloor\frac{c}{2}\rfloor,c}$,
or $\overline{G}\in \{W_{n,k,c}, ~Z_{n,k,c}\}$, where $\overline{G}$ denotes the $C$-closure of $G$.
\end{thm}

We give the formal reduction of Theorem \ref{Th:main-refined} to Theorems \ref{Th:BondyStab-2con} and \ref{Th:WZ+closure1}.

\medskip

{\noindent \it Proof of Theorem \ref{Th:main-refined}.} (Assuming Theorems \ref{Th:BondyStab-2con} and \ref{Th:WZ+closure1}.)
Let $G, C$ be as in Theorem \ref{Th:main-refined}.
We notice that $e(G)>f(n,\lfloor\frac{c}{2}\rfloor-1,c)=(\lfloor\frac{c}{2}\rfloor-1)(n-c)+h(c+1,\lfloor\frac{c}{2}\rfloor-1)$.
So either $e(G-C)+e(G-C,C)>(\lfloor\frac{c}{2}\rfloor-1)(n-c)$ or $e(G[C])> h(c+1,\lfloor\frac{c}{2}\rfloor-1)$.
If the former case occurs, then by Theorem \ref{Th:BondyStab-2con},
either $G\subseteq W_{n,\lfloor\frac{c}{2}\rfloor,c}$,
or $c$ is odd and $G$ is a subgraph of a member of $\mathcal{X}_{n,c}\cup \mathcal{Y}_{n,c}$.
As every graph in $\mathcal{X}_{n,c}\cup \mathcal{Y}_{n,c}$ has a vertex of degree 2, it is only valid when $k=2$.
So the latter case occurs.
Then the assertion of Theorem \ref{Th:main-refined} follows from Theorem \ref{Th:WZ+closure1}.
\qed

\subsection{Organization}
The rest of the paper is organized as follows.
In Section 2, first we introduce notations and terminologies,
which include an important concept `locally maximal cycle' for our proofs;
then we collect and prove some lemmas on cycles and closures.
In Section 3, we prove Theorem \ref{Th:BondyStab-2con}.
In Section 4, we prove a stronger version (Theorem \ref{Th:WZ+closure}) of Theorem \ref{Th:WZ+closure1},
whose proof will be split into three technical lemmas.
In Section 5, we complete the proofs of the two variants, i.e., Theorems \ref{Th:Var1} and \ref{Th:Var2}.
In Section 6, we conclude this paper by discussing some future research.




\section{Preliminaries}
\subsection{Notations}

Let $G$ be a graph and $H$ be a subgraph of $G$.
We use $G-H$ to denote the resulting graph obtained from $G$ by deleting all vertices of $H$.
If $H$ consists of only one vertex $v$, then we just write it as $G-v$.
For convenience, sometime we would abuse the notation by using the subgraph $H$ as its vertex set.
For instance, we often use $|H|$ to express $|V(H)|$.
Let $A$ be a subset of $V(G)$.
By $N_H(A)$, we mean the set of all vertices in $V(H)\backslash A$ which have at least one neighbor in $A$.
We write $G[A]$ for the induced subgraph of $G$ on $A$.
We say $A$ is {\it stable}, if $G[A]$ has no edges.
If $H, H'$ are two disjoint subgraphs (or subsets) in $G$,
we define $(H,H')$ to be the induced bipartite subgraph of $G$ on the two parts $V(H)$ and $V(H')$.
For $x,y\in V(G)$, an $(x,y)$-path is a path in $G$ with two endpoints $x$ and $y$,
and an $(x,H,y)$-path is an $(x,y)$-path with all internal vertices in $V(H)$.
We use $d^{*}_H(x,y)$ to denote the length of a longest $(x,H,y)$-path.
We say $G$ is {\it Hamiltonian-connected},
if for any two vertices $x,y\in V(G)$, there exists an $(x,y)$-path which passes through every vertex in $G$.
The {\it clique number} of $G$ is the maximum size of a clique in $G$.
For a cycle or path $C$ with a given orientation, we denote $v^+$ and $v^-$
as the successor and predecessor of the vertex $v$ on $C$, respectively.
For a subsect $A\subseteq V(C)$, by $A^+$ (resp. $A^-$)
we mean the set consisting of $v^+$ (resp. $v^-$) for all $v\in A$.
An $(x,y)$-path in $C$ sometime is also written as $C[x,y]$.
Two edges are {\it independent}, if their endpoints are distinct.


Let $C$ be a cycle of a graph $G$ and $R$ be a component of $G-C$. A subset $M=\{x_1,x_2,\ldots,x_s\}$ of $V(G)$
is called a \emph{\bf strong attachment} of $R$ to $C$, if $x_i$'s lie on $C$ in a cyclic order,
and for any ordered pair of vertices $x_i,x_{i+1}$, where $x_{s+1}=x_1$,
there exist $y_i,y_{i+1}\in V(R)$
such that $x_iy_i,x_{i+1}y_{i+1}$ are independent edges.

A cycle $C$ is \emph{\bf locally maximal} in a graph $G$ if
there is no cycle $C'$ in $G$ such that $|E(C')|>|E(C)|$ and $|E(C')\cap E(C,G-C)|\leq 2$.
This concept will play an important role in our proofs (for Section 4 especially).
It seems that in most situations a locally maximal cycle $C$ captures the properties of a longest cycle,
and yet it has its own advantages for counting the number of edges incident with $V(C)$.

Lastly, we consider the monotonicity of the function $f(n,k,c)$, where $n, c$ are fixed.
Basic calculation shows that $f(n,k,c)=\frac{3}{2}\left[k^2-(\frac{4c-2n}{3}+1)k\right]+\frac{c^2+c}{2}$ is convex in $k$.
So the maximum of $f(n,k,c)$ over an interval $[a,b]$ is always attained at either $k=a$ or $k=b$.
Assuming that $10\leq c\leq n-1$, we have
$\left\lfloor\frac{c}{2}\right\rfloor-1\geq \frac{c-1}{3}+\frac12\geq \frac{2c-n}{3}+\frac12$,
which implies that $f\left(n,\left\lfloor\frac{c}{2}\right\rfloor,c\right)\geq f\left(n,\left\lfloor\frac{c}{2}\right\rfloor-1,c\right)$.
This inequality will be needed in the proof of Theorem \ref{Th:Var1}.

\subsection{Some results on cycles}
We collect and prove some results on cycles here.
The following result is due to Bondy \cite{B71}, which strengthens Dirac's theorem.
\begin{thm}[Bondy \cite{B71}]\label{Le:DegreeBond}
	Let $G$ be a 2-connected graph on $n$ vertices. If every vertex except
	for at most one vertex is of degree at least $k$, then $c(G)\geq \min\{n,2k\}$.
\end{thm}

The next result, which was proved by Fan \cite{F90}, can be viewed as an average-degree version of the classical Erd\H{o}s-Gallai
theorem. This will be frequently used in our coming proofs for finding long paths between some specified vertices.

\begin{thm}[Fan \cite{F90}]\label{Le:FanErdoGal}
Let $x,y$ be two distinct vertices in a 2-connected graph $G$.
Suppose that the average degree of the vertices other than $x$ and $y$ in $G$ is $r$,
then the longest $(x,y)$-path in $G$ has length at least $r$,
with equality if and only if $r$ is an integer and $G\in \{J, J-xy\}$,
where $J$ denotes the union of some cliques $K_{r+1}$ which pairwise share the same vertices $x$ and $y$.
\end{thm}

One can derive the following lemma from Theorem 2 of \cite{F90} (choosing $k=2$).
\footnote{The original statement of Theorem 2 in \cite{F90} requires that
``$C$ is locally longest with respect to $H$ and $H$ is locally 2-connected to $C$'',
which can be implied if $C$ is a longest cycle in $G$ and both $G$ and $H$ are 2-connected as in Lemma \ref{Le:FanLongCyc}.}

\begin{lem}[Theorem 2, \cite{F90}]\label{Le:FanLongCyc}
Let $G$ be a 2-connected graph, $C$ be a longest cycle of length $c$ in $G$, and $H$ a component of $G-C$ which is 2-connected.
If the average degree of the vertices of $H$ in $G$ is $r$, then
$c\geq 2r$, with equality only if $H$ is a clique $K_{r-1}$ in which
every vertex has the same two neighbors on $C$.
\end{lem}



The following lemma studies some properties of a strong attachment.

\begin{lem}[Lemma 1, \cite{F90}]\label{Le:Fan90}
Let $G$ be a graph, $C$ be a cycle in $G$, and $R$
a component of $G-C$. Let $T=\{u_1,u_2,\ldots,u_t\}$ be a maximum
strong attachment of $R$ to $C$, $S=N_C(R)\backslash T$, $t=|T|$ and $s=|S|$. Then,\\
(i) Every vertex in $S$ is joined to only one vertex in $R$.\\
(ii) For each $1\leq i\leq t$, suppose
that $N_C(R)\cap V(C[u_i,u_{i+1}])=\{a_0,a_1,\ldots,a_p\}$,
where $a_0=u_i$, $a_{p}=u_{i+1}$, and $a_j$'s are in a cyclic order on $C$.
Then there is a subscript $m$ such that
$N_R(a_j)=N_R(a_0)$ for $0\leq j\leq m$
and
$N_R(a_j)=N_R(a_{p})$ for $m+1\leq j\leq p$.\\
(iii) If $C$ is a longest cycle in $G$ of length $c$ and $t\geq 2$,
then $c\geq \sum_{i=1}^td_{R}^{*}(u_i,u_{i+1})+2s$.
\end{lem}

Lastly, we bound the clique number on a long cycle by
some parameters related to a strong attachment.

\begin{lem}\label{Le:clique}
Let $G$ be a 2-connected graph, $C$ a locally maximal cycle in $G$, and $R$ a component of $G-C$.
Let $T$ be a strong attachment of $R$ to $C$.
Let $t=|T|, ~q=|N_C(R)\backslash T|$ and $\omega$ be the clique number of $G[C]$.
If for any $x,x'\in T$, the longest $(x,R,x')$-path
is of length at least $d$, where $d\geq 2$, then the following hold:\\
(i) $\omega\leq |C|-(d-1)(t-1)$;\\
(ii) If $T$ is a maximum strong attachment, then
$\omega\leq |C|-(d-1)(t-1)-q$.
\end{lem}

\begin{proof}
We write $C=x_1x_2...x_cx_1$ and view $x_1,x_2,...,x_c$ appearing on $C$ in the clockwise order.
(All subscripts are taken under modulo $c$ in this proof.)
For $x,y\in V(C)$, by $C[x,y]$ we denote the segment of $C$ from $x$ to $y$ in the clockwise order.
Let $T=\{u_1,u_2,\ldots, u_t\}$ and $u_j:=x_{i_j}$, where $1\leq i_1<\ldots<i_t\leq c$.
Let $W$ be a maximum clique in $G[C]$.

We first prove (i).
Since $d_R^*(u_j,u_{j+1})\geq d$ and $C$ is locally maximal,
we see that $C[u_j,u_{j+1}]$ is a path of length at least $d$.
For each $j$, let $A_j=V(C[x_{i_j+1},x_{i_j+\lceil\frac{d-1}{2}\rceil}])$
and $B_j=V(C[x_{i_{j+1}-\lfloor\frac{d-1}{2}\rfloor},x_{i_{j+1}-1}])$.
So $A_j$ and $B_j$ are disjoint.
Let $A=\cup_j A_j$ and $B=\cup_k B_k$.
We claim that for any $j\neq k$, there are no edges between $A_j$ and $A_k$.
Suppose this was not the case.
Then there exist $1\leq \ell, \ell'\leq \lceil\frac{d-1}{2}\rceil$
such that $e:=x_{i_j+\ell}x_{i_k+\ell'}\in E(G)$.
Let $P$ be a $(u_j,R,u_k)$-path of length at least $d$.
Then $C[x_{i_k+\ell'},x_{i_j}]\cup P\cup C[x_{i_j+\ell}, x_{i_k}]\cup \{e\}$
forms a cycle, say $C'$.
We see that $|C'|\geq |C|+(d+1)-(\ell+\ell')\geq |C|+(d+1)-2\lceil\frac{d-1}{2}\rceil>|C|$
and $|E(C')\cap E(C,G-C)|\leq 2$, a contradiction to that $C$ is locally maximal,
proving the claim.
The claim shows that the maximum clique $W$ can intersect with at most one $A_j$,
so $|V(W)\cap A|\leq \lceil\frac{d-1}{2}\rceil$.
Similarly, there are no edges between $B_j$ and $B_k$ for any $j\neq k$,
and thus $|V(W)\cap B|\leq \lfloor\frac{d-1}{2}\rfloor$.
As $A$ and $B$ are disjoint, we have $|A\cup B|=(d-1)t$.
Now we prove (i) by showing
\begin{align*}
\omega&=|V(W)\backslash (A\cup B)|+|V(W)\cap (A\cup B)|\leq |V(C)\backslash (A\cup B)|+|V(W)\cap (A\cup B)|\\
&\leq |V(C)|-|A\cup B|+|V(W)\cap A|+|V(W)\cap B|\leq c-(d-1)(t-1).
\end{align*}

To prove (ii), we need a refined argument for (i).
For any $1\leq j\leq t$,
let $N_C(R)\cap V(C[u_j,u_{j+1}])=\{a_0,a_1,\ldots,a_p\}$,
where $a_0=u_j$, $a_{p}=u_{j+1}$, and $a_\ell$'s for $0\leq \ell\leq p$ appear on $C$ in the clockwise order.
By Lemma \ref{Le:Fan90},
there is a subscript $m$ such that $N_R(a_\ell)=N_R(u_j)$ for $0\leq \ell\leq m$ and $N_R(a_\ell)=N_R(u_{j+1})$
for $m+1\leq \ell\leq p$.
Consider the segment $C_j:=C[a_m, a_{m+1}]$.
Since $N_R(a_m)=N_R(u_j)$ and $N_R(a_{m+1})=N_R(u_{j+1})$,
we see that $d_R^*(a_m, a_{m+1})=d_R^*(u_j, u_{j+1})\geq d$.
So $C_j$ is a path of length at least $d$.
Similarly as in the proof of (i),
let $A_j$ be the set of the first $\lceil\frac{d-1}{2}\rceil$ vertices on $C_j$ (starting from $a_m$ but not including $a_m$),
and let $B_j$ be the set of the last $\lfloor\frac{d-1}{2}\rfloor$ vertices on $C_j$ (not including $a_{m+1}$).
Also let $A=\cup_j A_j$ and $B=\cup_j B_j$.
So $A$ and $B$ are disjoint.
Similarly, we can show that $W$ intersects with at most one $A_j$.
Thus, $|V(W)\cap A|\leq \lceil\frac{d-1}{2}\rceil$, with equality if and only if
$V(W)\cap A=A_j$ for some $j$.
Also $|V(W)\cap B|\leq \lfloor\frac{d-1}{2}\rfloor$.

We consider $(N_C(R))^+$.
First we show $(N_C(R))^+$ is stable. Otherwise, there exist $x,y\in N_C(R)$ with  $x^+y^+\in E(G)$;
let $P$ be any $(x,R,y)$-path, which has length at least 2, then $(C-\{xx^+,yy^+\})\cup P\cup \{x^+y^+\}$ is a cycle
contradicting that $C$ is locally maximal.
We point out that
$(N_C(R))^+$ is disjoint from $B$, and it intersects with each $A_j$ in exactly one vertex $a_m^+$ (i.e., the first vertex after $a_m$ in $C_j$).
Let $D:=(N_C(R))^+\backslash (A\cup B)$, where $|D|=|N_C(R)|-t=q$.
We claim that $|V(W)\cap (A\cup D)|\leq \lceil\frac{d-1}{2}\rceil$.
Since $D$ is stable, $W$ intersects with $D$ in at most one vertex.
If $V(W)\cap D=\emptyset$, then this claim follows from that $|V(W)\cap A|\leq \lceil\frac{d-1}{2}\rceil$.
So we may assume that $V(W)\cap D=\{x\}$ and $|V(W)\cap A|=\lceil\frac{d-1}{2}\rceil$.
By the above analysis, we then have $V(W)\cap A=A_j$ for some $j$.
In particular, the vertex $a_m^+\in A_j$ is in $W$.
But then there are two vertices $x,a_m^+$ in $V(W)\cap (N_C(R))^+$, a contradiction.
This proves the claim.
Combining the above bounds, we have
\begin{align*}
\omega&\leq |V(C)\backslash (A\cup B\cup D)|+|V(W)\cap (A\cup B\cup D)|\\
&\leq |V(C)|-|A\cup B\cup D|+(d-1)\leq c-(d-1)(t-1)-q.
\end{align*}
This finishes the proof of Lemma \ref{Le:clique}.
\end{proof}

\subsection{Lemmas on closures}
In this subsection, we prove some lemmas on $C$-closures.

\begin{lem}\label{Le:ClosLongPath}
Let $G$ be a graph on $n$ vertices.
Then for any $x,y\in V(G)$,
the longest $(x,y)$-path in the $(n+1)$-closure of $G$ has the same length as the longest $(x,y)$-path in $G$.
\end{lem}
\begin{proof}
It suffices to prove the following:
for two nonadjacent vertices $u,v$ in $G$ with $d_G(u)+d_G(v)\geq n+1$ and for any $x,y\in V(G)$,
there exists a longest $(x,y)$-path $P$ in $G':=G+\{uv\}$ satisfying that $E(P)\subseteq E(G)$.
Suppose this is not true.
Then any longest $(x,y)$-path $P$ in $G'$ must contain the new edge $uv$.
Assume that $x,u,v,y$ lie on $P$ in this order.
First we observe that there is no common neighbor of $u$ and $v$ in $V(G)-V(P)$,
as otherwise one can find an $(x,y)$-path longer than $P$ in $G'$.
Let $P_1:=P[x,u]$ and $P_2:=P[v,y]$.
We claim that there are no vertices $a\in N_G(u)\cap V(P_1)$ and $b\in N_G(v)\cap V(P_1)$
such that $b=a^+$ (we view $P$ from $x$ to $y$).
Suppose such $a,b$ exist.
Then $b\in V(P_1)\backslash \{u\}$.
By Posa's rotation technique, $(P-\{ab,uv\})\cup \{au,bv\}$ is a longest $(x,y)$-path in $G'$,
however all its edges are from $E(G)$, a contradiction.
This shows that $(N_G(u)\cap V(P_1))^+\cap (N_G(v)\cap V(P_1))=\emptyset$.
So $|N_G(u)\cap V(P_1)|+|N_G(v)\cap V(P_1)|\leq |V(P_1)|$.
Similarly, we have $|N_G(u)\cap V(P_2)|+|N_G(v)\cap V(P_2)|\leq |V(P_2)|$.
Combining the above bounds, it follows that $d_{G}(u)+d_{G}(v)\leq |V(G)|=n$,
contradicting that $d_{G}(u)+d_{G}(v)\geq n+1$. This proves the lemma.
\end{proof}

\begin{lem}\label{Le:Coper}
Let $G$ be a graph and $C$ be a locally maximal cycle of $G$.
Then $C$ is also a locally maximal cycle of the $C$-closure of $G$.
\end{lem}

\begin{proof}
Let $\overline{G}$ denote the $C$-closure of $G$.
We point out that $\overline{G}-C=G-C$ and $E(C,\overline{G}-C)=E(C,G-C)$.
Suppose this is not true. Then there is a cycle $D$ in $\overline{G}$ such that $|D|>|C|$ and
$|E(D)\cap E(C,\overline{G}-C)|\leq 2$; and subject to this, we choose $D$ such that $|D|$ is maximum.
It is fair to assume that $|E(D)\cap E(C,\overline{G}-C)|=2$
(as otherwise $E(D)\cap E(C,\overline{G}-C)=\emptyset$, implying that $D\subseteq G-C$).
Let $xy,x'y'$ be the two edges in the intersection, where $x,x'\in V(C)$.
Then $D$ consists of two internally disjoint $(x,x')$-paths $P_1$ and $P_2$,
where $P_1$ is an $(x,G-C,x')$-path and $P_2$ is an $(x,x')$-path in $\overline{G}[C]$.
Note that $P_1$ is in $G$,
and by the maximality of $D$, $P_2$ is a longest $(x,x')$-path in $\overline{G}[C]$.
By Lemma \ref{Le:ClosLongPath}, there exists an $(x,x')$-path $P_3$ in $G[C]$ with $|P_2|=|P_3|$.
Set $C':=P_1\cup P_3$. Then $C'$ is a cycle in $G$ of length
$|C'|=|P_1|+|P_3|=|P_1|+|P_2|=|D|>|C|$. Furthermore,
$|E(C')\cap E(C,G-C)|=|E(D)\cap E(C,\overline{G}-C)|\leq 2$, which contradicts
the fact that $C$ is locally maximal in $G$. This completes the proof.
\end{proof}

\begin{lem}\label{Le:CoperHC}
Let $G$ be a 2-connected graph on $n$ vertices
and $C$ be a locally maximal cycle in $G$ of length $c$, where $c\leq n-1$.
Let $\overline{G}$ denote the $C$-closure of $G$.
Then $\overline{G}[C]$ is non-Hamiltonian-connected.
\end{lem}
\begin{proof}
Suppose for a contradiction that $\overline{G}[C]$ is Hamiltonian-connected.
As $c\leq n-1$, there is a component $R$ in $G-C$. Since $G$ is 2-connected,
there exist two distinct vertices $x,x'\in N_C(R)$.
Let $P_1$ be an $(x,R,x')$-path.
Since $\overline{G}[C]$ is Hamiltonian-connected, there is an $(x,x')$-path $P_2$ in $\overline{G}[C]$,
which passes through all vertices in $V(C)$. Then $C':=P_1\cup P_2$ is a cycle in $\overline{G}$
which is longer than $C$ and $|E(C')\cap E(C,\overline{G}-C)|\leq 2$.
This contradicts Lemma \ref{Le:Coper} that $C$ is locally maximal in $\overline{G}$.
This proves the lemma.
\end{proof}

We need a theorem of Chv\'{a}tal \cite{C72} on the degree sequences of non-Hamiltonian graphs.

\begin{thm}[Chv\'{a}tal \cite{C72}]\label{Th:chvatal}
Let $G$ be a graph with degree sequence $d_1\leq d_2\leq ... \leq d_n$ and $n\geq 3$.
If $G$ is non-Hamiltonian, then there is some integer $s< n/2$
such that $d_s\leq s$ and $d_{n-s}< n-s$.
\end{thm}

We can get a corollary of Chv\'{a}tal's theorem on non-Hamiltonian-connected graphs.

\begin{lem}\label{Le:HamilCon-Posa}
Let $G$ be a non-Hamiltonian-connected graph on $n$ vertices with minimum degree at least 2.
Then there exists a set of $s-1$ vertices in $G$ of degree at most $s$,
where $2\leq s\leq \lfloor\frac{n}{2}\rfloor$.
\end{lem}

\begin{proof}
Since $G$ is non-Hamiltonian-connected,
there exist $x,y\in V(G)$ such that there is no Hamiltonian path from $x$ to $y$ in $G$.
Let $G'$ be obtained from $G$ by adding a new vertex $z$ and two edges $xz,yz$.
Clearly, $G'$ is not Hamiltonian.
Let $d_1\leq d_2\leq \ldots \leq d_{n+1}$ be the degree sequence of $G'$.
Since $\delta(G)\geq 2$, we have $d_1=2$, which denotes the degree of $z$ in $G'$.
By Theorem \ref{Th:chvatal}, there is some integer $s<\frac{n+1}{2}$ such that $d_s\leq s$ and $d_{n+1-s}< n+1-s$.
As $d_1=2$, we see that $2\leq s\leq \lfloor\frac{n}{2}\rfloor$.
If we let $f_1\leq f_2\leq \ldots \leq f_n$ be the degree sequence of $G$,
then each $f_i$ corresponds to the vertex associated with $d_{i+1}$ and thus $f_i\leq d_{i+1}$.
This shows that $f_{s-1}\leq d_s\leq s$, proving the lemma.
\end{proof}

The next lemma (in particular, its special case when $\delta=1$) will play an important role in the proof of Theorem \ref{Th:WZ+closure}.
We establish a general version for possible studies in future.
Its proof is analogous to Lemma 6 in \cite{FKL17}.

\begin{lem}\label{Le:HamConStruc2}
Let $G_c$ be a graph on $c$ vertices with minimum degree at least 2.
Further suppose that $G_c$ is $(c+1)$-closed and non-Hamiltonian-connected
with $$e(G_c)>h\left(c+1,\left\lfloor\frac{c}{2}\right\rfloor-p\right) \text{ for some integer } p\geq 0.$$
Then one of the following holds:\\
(i)  $G_c$ contains a subset $S$ of $s-1$ vertices of degree at most
$s$, where $2\leq s\leq \lfloor\frac{c}{2}\rfloor-p-1$, such that
$G_c-S$ is a clique; or \\
(ii) $G_c$ contains a subset $T$ of $t-1$ vertices of degree at most
$t$, where $\lfloor\frac{c}{2}\rfloor-p+1\leq t\leq \lfloor\frac{c}{2}\rfloor$.
\end{lem}

\begin{proof}
Suppose neither (i) nor (ii) holds. Since $G_c$ is non-Hamiltonian-connected,
by Lemma \ref{Le:HamilCon-Posa}, there exists some
$2\leq s\leq \lfloor\frac{c}{2}\rfloor$ such that $G_c$ contains $s-1$ vertices
of degree at most $s$. Subject to this, we choose $s$ to be maximal, and let
$S$ be the set of all vertices in $G_c$ with degree at most $s$.
If $\lfloor\frac{c}{2}\rfloor-p+1\leq s\leq \lfloor\frac{c}{2}\rfloor$,
then $(ii)$ holds.
If $s=\lfloor\frac{c}{2}\rfloor-p$,
then $e(G_c)\leq (s-1)s+\binom{c-s+1}{2}=h(c+1,s)=h(c+1,\lfloor\frac{c}{2}\rfloor-p)$,
a contradiction.
So we may assume that $2\leq s\leq\lfloor\frac{c}{2}\rfloor-p-1$.
Moreover, by the maximality of $s$, we have $|S|=s-1$.

Next, we will show that $G_c-S$ is a clique.
Suppose that there are nonadjacent vertices $u,v\in V(G_c)-S$. Without loss of generality, assume that $u$ is the one with the maximal
degree among all vertices in $V(G_c)-S$, each of which is not adjacent to every vertex in $V(G_c)-S$.
Let $S':=V(G)-N(u)-\{u\}$ and $s':=|S'|+1=c-d(u)$.
For any $w\in S'$, since $wu\notin E(G_c)$ and $G_c$ is $(c+1)$-closed,
we have $d(w)\leq c-d(u)=s'$. So $S'$ is a set of $s'-1$ vertices of degree at most $s'$.
Since $v\notin S$, by the maximality of $S$, it follows that $d(v)>s$.
So $s<d(v)\leq s'$.
By the maximality of $s$, we get that $s'\geq \lfloor\frac{c}{2}\rfloor+1$.
As $s'=c-d(u)$, we get $d(u)\leq \lfloor\frac{c}{2}\rfloor$.
We then claim that any vertex $x\in S'$ has degree at most $\lfloor\frac{c}{2}\rfloor$.
Indeed, if $x\in S$ then $d(x)\leq s\leq \lfloor\frac{c}{2}\rfloor-p-1$;
otherwise $x\in S'\backslash S$, then by the choice of $u$, $d(x)\leq d(u)\leq \lfloor\frac{c}{2}\rfloor$.
Now observe that $S'$ is a set of at least $\lfloor\frac{c}{2}\rfloor$ vertices of degree at most $\lfloor\frac{c}{2}\rfloor$,
so $(ii)$ holds, a contradiction. This shows that $G_c-S$ is a clique and thus $(i)$ holds.
This proves the lemma.
\end{proof}

We remark that if $p=0$, then only (i) occurs in Lemma \ref{Le:HamConStruc2}.

\section{Stability on a theorem of Bondy}
In this section, we prove a stability result on a classic theorem due to Bondy \cite{B71}.
We restate the statement here for the convenience of the readers.

\medskip

\noindent {\bf Theorem \ref{Th:BondyStab-2con}.}
{\em
	Let $G$ be a 2-connected graph on $n$ vertices and $C$ be a longest cycle in $G$ of length $c$, where $10\leq c\leq n-1$.
	If $
    e(G-C)+e(G-C,C)>\left(\left\lfloor\frac{c}{2}\right\rfloor-1\right)(n-c),
    $
    then either $G\subseteq W_{n,\lfloor\frac{c}{2}\rfloor,c}$ or
	$c$ is odd and $G$ is a subgraph of a member of $\mathcal{X}_{n,c}\cup \mathcal{Y}_{n,c}$.
}

\medskip

To prove Theorem \ref{Th:BondyStab-2con}, a crucial step is to find a vertex in $G-C$ with $\lfloor\frac{c}{2}\rfloor$ neighbors in $C$
(see Theorem \ref{Th:BondyStab-spec}); this will be done in Subsection \ref{subsec:Bondy1}.
As we shall see later (somehow surprisingly),
the existence of such a vertex can give a lot of structural information of the graph $G$.
We then complete the proof of Theorem \ref{Th:BondyStab-2con} in Subsection \ref{subsec:Bondy2}.


\subsection{A vertex with large degree}\label{subsec:Bondy1}
In this subsection, we prove the following result.
\begin{thm}\label{Th:BondyStab-spec}
	Let $G$ be a 2-connected graph on $n$ vertices and $C$ be a longest cycle in $G$ of length $c$, where $10\leq c\leq n-1$.
	If $e(G-C)+e(G-C,C)>(\lfloor\frac{c}{2}\rfloor-1)(n-c),$
	then there exists an isolated vertex $u$ in $G-C$ with $d_C(u)=\lfloor\frac{c}{2}\rfloor$.
\end{thm}

Just as in the original theorem of Bondy, we also can drop off the connectivity condition.
A more general statement is as follows.

\begin{thm}\label{Th:BondyStab}
Let $G$ be a graph on $n$ vertices and $C$ be a longest cycle in $G$ of length $c$, where $4\leq c\leq n-1$.
If $e(G-C)+e(G-C,C)>(\lfloor\frac{c}{2}\rfloor-1)(n-c)$,
then one of the following holds:\\
(a) $c\in\{6,7,9\}$;\\
(b) there exists a vertex $u\in V(G-C)$ with $d_C(u)=\lfloor\frac{c}{2}\rfloor$;\\
(c) there exists a cycle $C'$ in $G$ satisfying that $|V(C\cap C')|\leq 1$ and \\
-- if $V(C\cap C')=\emptyset$, then $|C'|\geq 2\lfloor\frac{c}{2}\rfloor-3$,\\
-- if $|V(C\cap C')|= 1$, then $|C'|\geq 2\lfloor\frac{c}{2}\rfloor-1$.
\end{thm}

\begin{proof}
We prove the theorem by contradiction. Suppose that there exists an $n$-vertex non-Hamiltonian graph $G$ and a longest cycle $C$ in $G$ of length $c\geq 4$ such that $e(G-C)+e(G-C,C)>(\lfloor\frac{c}{2}\rfloor-1)(n-c)$ and none of $(a), (b)$ and $(c)$ holds.
We choose such a counterexample $G$ that $c$ is minimum and subject to this, the order $n$ is minimum.
Throughout this proof, let $H:=G-C$ and so
\begin{align}\label{equ:Bondy}
e(H)+e(H,C)>\left(\left\lfloor\frac{c}{2}\right\rfloor-1\right)(n-c).
\end{align}

\setcounter{claim}{0}
\begin{claim}
	$c\geq 5$ and $n\geq c+2$.
\end{claim}
\begin{proof}
Assume that $c=4$. Then by \eqref{equ:Bondy} we have $e(H)+e(H,C)>n-4$.
Suppose that there is a cycle $C'$ in $G-E(C)$. So $|C'|=3$ or $4$.
If $|V(C'\cap C)|\leq 1$, then clearly $(c)$ holds;
otherwise $|V(C'\cap C)|\geq 2$, then there exists either a cycle longer than $C$ or
a vertex in $H$ with two neighbors in $C$ (thus $(b)$ holds), a contradiction.
So there is no cycle in $G-E(C)$.
Consider any component $R$ in $H$, which must be a tree.
If $d_C(R)\geq 2$, then either there is a vertex in $R$ with two neighbors in $C$, or we can find a longer cycle, a contradiction.
Thus $d_C(R)\leq 1$ and as $G-E(C)$ has no cycles, $e(R,C)\leq 1$.
This implies that $e(R)+e(R,C)\leq |R|$.
Summing over all components $R$ in $H$, we have $e(H)+e(H,C)\leq\sum |R|=n-4$, a contradiction.
This proves that $c\geq 5$.
	
Now suppose that $n=c+1$. Let $V(H)=\{u\}$. By \eqref{equ:Bondy}, $d_C(u)\geq \lfloor\frac{c}{2}\rfloor$.
Since $C$ is a longest cycle in $G$, we must have $d_C(u)=\lfloor\frac{c}{2}\rfloor$. This proves Claim 1.
\end{proof}

\begin{claim}
	For any vertex $v\in V(H)$, $d_G(v)\geq \lfloor\frac{c}{2}\rfloor$.
\end{claim}
\begin{proof}
	Suppose for a contradiction that there exists a vertex $u\in V(H)$ with $d_G(u)\leq \lfloor\frac{c}{2}\rfloor-1$.
	Set $G':=G-u$. So $C$ remains a longest cycle in $G'$. By Claim 1, $n\geq c+2$,
	implying that $G'$ is non-Hamiltonian. We also have
	$e(G'-C)+e(G'-C,C)=e(H)+e(H,C)-d_G(u)>(\lfloor\frac{c}{2}\rfloor-1)(n-c-1)$.
	By the choice of $G$, one of $(a)$, $(b)$ and $(c)$ holds in $G'$.
    It is obvious to see that the same case also holds in $G$. This proves the claim.
\end{proof}

\begin{claim}
	$H$ is connected.
\end{claim}
\begin{proof}
	Suppose that $H$ is not connected. Then by averaging, there exists
	a component $R$ in $H$ such that $e(G[R])+e(G[R],C)>(\lfloor\frac{c}{2}\rfloor-1)\cdot |R|$.
	It is clear that $G[R\cup C]$ is non-Hamiltonian. Then by the choice of $G$,
	one of $(a),(b)$ and $(c)$ holds in $G[R\cup C]$, which also holds in $G$, a contradiction.
	This proves Claim 3.
\end{proof}

\begin{claim}
	$G$ is 2-connected.
\end{claim}

\begin{proof}
Suppose that $G$ is not 2-connected.
Then there exists an end-block\footnote{A {\it block} $B$ in a graph $G$ is a maximal connected subgraph of $G$
such that there exists no cut-vertex of $B$. An {\it end-block} in $G$ is a block in $G$ containing at most one cut-vertex of $G$.}
$B$ of $G$ such that $|V(B\cap C)|\leq 1$.
Let $b$ be the unique cut-vertex of $G$ with $b\in V(B)$ (if it exists).
By Claim 2, every vertex in $V(B)$, except the vertex $b$, has degree at least $\lfloor\frac{c}{2}\rfloor$ in $B$.
By Theorem \ref{Le:DegreeBond}, we have $c(B)\geq \min\{|B|, 2\lfloor\frac{c}{2}\rfloor\}$.
If $|B|\geq 2\lfloor\frac{c}{2}\rfloor-1$, then $c(B)\geq 2\lfloor\frac{c}{2}\rfloor-1$, a contradiction to $(c)$.
Hence we may assume that $|B|\leq 2\lfloor\frac{c}{2}\rfloor-2$.

    Let $H_1:=G-(B-b)$. Clearly $C$ is still a longest cycle in $H_1$.
    We claim that $H_1$ is not Hamiltonian.
    Indeed, otherwise $C$ must be a Hamiltonian cycle of $H_1$ and thus we have $H=B-b$.
    So $\frac{(|B|-1)|B|}{2}\geq e(B)=e(H)+e(H,C)>(\lfloor\frac{c}{2}\rfloor-1)(n-c)=(\lfloor\frac{c}{2}\rfloor-1)(|B|-1)$,
	which implies that $|B|\geq 2\lfloor\frac{c}{2}\rfloor-1$,
	a contradiction.

    Note that we have $e(H_1)=e(G)-e(B)$.
    So $e(H_1-C)+e(H_1-C,C)= e(H_1)-e(G[C])=e(G-C)+e(G-C,C)-e(B)>(\lfloor\frac{c}{2}\rfloor-1)(n-c)-\frac{|B|(|B|-1)}{2}$.
    Since $|B|\leq 2\lfloor\frac{c}{2}\rfloor-2$, it follows that
    $e(H_1-C)+e(H_1-C,C)>(\lfloor\frac{c}{2}\rfloor-1)(|V(H_1)|-c)$.
	By the choice of $G$, one of $(a)$, $(b)$ and $(c)$ holds in $H_1$, which also holds in $G$.
    This proves Claim 4.
\end{proof}

\begin{claim}
	$|V(H)|\geq 3$.
\end{claim}
\begin{proof}
Otherwise, in view of Claims 1 and 3, we may assume that $H$ is just an edge $v_1v_2$.
So we have $e(H,C)\geq 2(\lfloor\frac{c}{2}\rfloor-1)$.
Let $T=\{u_1,u_2,\ldots,u_t\}\subset N_C(H)$ be a maximum strong attachment of $H$ to $C$.
Let $S:=N_C(H)\backslash T$, $t=|T|$ and $s=|S|$.
For any $u_i,u_{i+1}\in T$, the $(u_i,H,u_{i+1})$-path is of at least length 3.
By Lemma \ref{Le:Fan90}, we have $e(H,C)\leq 2t+s$ and $c\geq 3t+2s\geq \frac{3}{2}e(H,C)\geq 3(\lfloor\frac{c}{2}\rfloor-1)$.
From this, we can derive a contradiction if $c\geq 8$ is even or $c\geq 11$ is odd.
Thus, $c\in \{5,6,7,9\}$. It only needs to consider $c=5$, as otherwise $(a)$ holds.
In case of $c=5$, we have $e(H,C)\geq 2$ and as $G$ is 2-connected, there are two independent edges in $(H,C)$,
which would lead to a cycle of length at least 6, a contradiction. This proves Claim 5.
\end{proof}

\begin{claim}
$H$ is 2-connected.
\end{claim}
\begin{proof}
Suppose that $H$ is not 2-connected. As $|V(H)|\geq 3$,
there exist two end-blocks $B_1, B_2$ of $H$.
Let $b_i$ be the unique cut-vertex of $H$ with $b_i\in V(B_i)$ for $i\in \{1,2\}$.
Since $G$ is 2-connected, there exists a vertex $v_2\in V(B_2-b_2)$ with a neighbor $u_2\in V(C)$.
	
First assume that $B_1$ is an edge, say $v_1b_1$.
By Claim 2, we have $d_C(v_1)\geq \lfloor\frac{c}{2}\rfloor-1$.
If $d_C(v_1)\geq \lfloor\frac{c}{2}\rfloor$, then $(b)$ holds, a contradiction.
Thus, $d_C(v_1)=\lfloor\frac{c}{2}\rfloor-1$.
Notice that there is a $(v_1,v_2)$-path in $H$ of length at least 2.
If $u_2$ is the unique neighbor of $v_1$ on $C$,
then we have $c=5$ and there exists a cycle in $G[H\cup \{u_2\}]$ of length at least 4, a contradiction to $(c)$.
Hence, we may assume that $N_C(v_1)\backslash \{u_2\}\neq \emptyset$.
Let $w_1,w_2,...,w_t$ be the neighbors of $v_1$ on $C$ which appear in a cyclic order,
where $t=d_C(v_1)=\lfloor\frac{c}{2}\rfloor-1$.
For any $w_i\in N_C(v_1)\backslash \{u_2\}$,
since there exists a $(w_i,H,u_2)$-path of length at least 4,
any $(w_i,u_2)$-segment of $C$ has length at least 4. So we have $c\geq 8$ and thus $t\geq 3$.
Let $w_i, w_j\in N_C(v_1)\backslash \{u_2\}$ be two vertices such that $u_2$ is contained in a $(w_i,w_j)$-segment $P$ of $C$
and subject to this, $P$ is minimal.
Since $P$ is a union of a $(w_i,u_2)$-segment and a $(w_j,u_2)$-segment of $C$, we get that $|P|\geq 8$.
There are at least $t-2=\lfloor\frac{c}{2}\rfloor-3$ segments
between two consecutive $w_\ell,w_{\ell+1}$ in $C-E(P)$, each of which has length at least 2.
So $c=|C|\geq |P|+2(\lfloor\frac{c}{2}\rfloor-3)\geq 8+2(\lfloor\frac{c}{2}\rfloor-3)\geq c+1$, a contradiction.

Now suppose that $|V(B_1)|\geq 3$. So $B_1$ is 2-connected.
Let $d:=c(B_1)$ and $r:=|B_1|-1$.
By the Erd\H{o}s-Gallai theorem (Theorem \ref{Th:EG-cyc}),
we have $e(B_1)\leq \frac{dr}{2}$.
If $d\geq 2\lfloor\frac{c}{2}\rfloor-3$, then $(c)$ holds.
So we have $d\leq 2\lfloor\frac{c}{2}\rfloor-4.$

We claim that $e(B_1-b_1, C)>(\lfloor\frac{c}{2}\rfloor-1-\frac{d}{2})r$.
Suppose for a contradiction that $e(B_1-b_1, C)\leq (\lfloor\frac{c}{2}\rfloor-1-\frac{d}{2})r$.
Consider $G_1:=G-(B_1-b_1)$. Since $e(B_1)\leq \frac{dr}{2}$,
$e(H)=e(B_1)+e(G_1-C)$ and $e(H,C)=e(B_1-b_1,C)+e(G_1-C,C)$, we have
\begin{align*}
e(G_1-C)+e(G_1-C,C)&=e(H)+e(H,C)-e(B_1)-e(B_1-b_1, C)\\
&>\left(\left\lfloor\frac{c}{2}\right\rfloor-1\right)\cdot (n-r-c)
=\left(\left\lfloor\frac{c}{2}\right\rfloor-1\right)\cdot (|V(G_1)|-c).
\end{align*}
As $G_1$ is not Hamiltonian (because $|G_1|>|C|$),
by the choice of $G$, we see that one of $(a), (b)$ and $(c)$ holds in $G_1$ and thus in $G$, a contradiction.

Therefore by averaging, there exists a vertex $v_1\in V(B_1-b_1)$
with $t:=d_C(v_1)\geq \lfloor\frac{c}{2}\rfloor-\frac{d}{2}$.
As we have $d\leq 2\lfloor\frac{c}{2}\rfloor-4$, it follows that $t\geq \lfloor\frac{c}{2}\rfloor-\frac{d}{2}\geq 2.$
Let $w_1,w_2,...,w_t$ be the neighbors of $v_1$ on $C$ which appear in a cyclic order.
Since $c(B_1)=d$ and $B_1$ is 2-connected,
$B_1$ contains a $(v_1,b_1)$-path of length at least $\lceil \frac{d}{2}\rceil$.
So for each $w_\ell\in N_C(v_1)\backslash \{u_2\}$, there exists a $(w_\ell,H,u_2)$-path of length at least $\lceil \frac{d}{2}\rceil+3$,
which in turn implies that any $(w_\ell,u_2)$-segment of $C$ has length at least $\lceil \frac{d}{2}\rceil+3$.
This shows that if $t=2$ (and thus $\frac{d}{2}+2\geq \lfloor\frac{c}{2}\rfloor$),
then $c\geq 2(\lceil \frac{d}{2}\rceil+3)\geq 2(\lfloor\frac{c}{2}\rfloor+1)\geq c+1$, a contradiction.
So we have $t\geq 3$.
Let $w_i, w_j\in N_C(v_1)\backslash \{u_2\}$ be two vertices
such that $u_2$ is contained in a $(w_i,w_j)$-segment $P$ of $C$
and subject to this, $P$ is minimal.
Since $P$ is a union of a $(w_i,u_2)$-segment and a $(w_j,u_2)$-segment of $C$,
we have $|P|\geq 2(\lceil \frac{d}{2}\rceil+3)$.
Also there are at least $t-2$ segments between two consecutive $w_\ell,w_{\ell+1}$ in $C-E(P)$,
each of which has length at least 2.
So $c\geq 2(t-2)+2(\lceil \frac{d}{2}\rceil+3)\geq 2\lfloor \frac{c}{2}\rfloor +2\geq c+1$, a contradiction. This proves Claim 6.
\end{proof}

We now distinguish between the parities of $c$.
First assume that $c$ is even.
By Claim 2, the average degree of $H$ in $G$ is at least $\frac{c}{2}$.
By Lemma \ref{Le:FanLongCyc}, either $|C|\geq c+1$, or $H$ is a complete graph $K_{\frac{c}{2}-1}$ in which
every vertex has the same two neighbors on $C$.
Thus, the latter case occurs.
Then we have $\binom{\frac{c}{2}-1}{2}+c-2=e(H)+e(H,C)>(\frac{c}{2}-1)(n-c)=(\frac{c}{2}-1)^2$,
which implies that $2<c<8$. As $c\geq 5$ is even, we have $c=6$.
However, in this case $H$ becomes a $K_2$, a contradiction to Claim 5.
	
In what follows we consider the case that $c$ is odd.
Set $p:=|N_C(H)|$. By Claims 2 and 6, every vertex $v\in V(H)$ has
at least $\mu:=\max\{\lfloor\frac{c}{2}\rfloor-p,2\}$ neighbors in $H$.
Let $T=\{u_1,...,u_t\}$ be a maximum attachment of $H$ to $C$, and $S:=N_C(H)\backslash T$,
where $t\geq 2$, $s:=|S|$ and $p=s+t$.
By Theorem \ref{Le:FanErdoGal},
there exists a $(u_i,H,u_{i+1})$-path of length at least $\mu+2$.
If $t\geq 3$, then by Lemma \ref{Le:Fan90}, we have
$$c\geq t(\mu+2)+2s=(t-2)\mu+2(\mu+s+t)\geq 2+2\lfloor\frac{c}{2}\rfloor\geq c+1,$$
a contradiction.
Now we only need to consider the case $t=2$.

Let $T=\{u_1,u_2\}$ and $v_1u_1,v_2u_2$ be two independent edges for some $v_1, v_2\in V(H)$.
By Lemma \ref{Le:Fan90}, every vertex in $S$ has a unique neighbor in $H$,
which is either $v_1$ or $v_2$.
This shows that for any $u\in V(H)\backslash\{v_1,v_2\}$, $d_H(u)\geq \frac{c-1}{2}-2$.
By Theorem \ref{Le:FanErdoGal}, there is a $(v_1,v_2)$-path in $H$ of length $\ell\geq \frac{c-1}{2}-2$.
So we have a $(u_1,H,u_2)$-path of length $\ell+2$.
If $|S|\geq 1$ or $\ell\geq \frac{c-1}{2}-1$, then by Lemma \ref{Le:Fan90}, $c=|C|\geq 2(\ell+2)+2|S|\geq c+1$, a contradiction.
So $S=\emptyset$ and $\ell=\frac{c-1}{2}-2$.
This implies that every $u\in V(H)\backslash\{v_1,v_2\}$
has $d_H(u)=\frac{c-1}{2}-2$ and thus is adjacent to both of $u_1,u_2$.

As $S=\emptyset$, it also holds that $\delta(H)\geq \frac{c-1}{2}-2$.
Since $H$ is 2-connected, Dirac's theorem \cite{D52} shows that $c(H)\geq \min\{|H|,2\delta(H)\}\geq \min\{|H|,c-5\}$.
If $c(H)\geq c-4$, then $(c)$ holds. So we have $3\leq c(H)\leq c-5$ (note that this shows $c\geq 8$).
This implies that either $c(H)=c-5$, or $c(H)=|H|\leq c-6$.
If the latter case holds, then
$e(H)+e(H,C)\leq \frac{(n-c)(n-c-1)}{2}+2(n-c)\leq \frac{(n-c)(c-7)}{2}+2(n-c)=\frac{c-3}{2}(n-c)$,
a contradiction to \eqref{equ:Bondy}.
So we have $c(H)=|H|=c-5$.
Recall that $u_1,u_2$ are adjacent to all vertices in $H-\{v_1,v_2\}$ and $u_1v_1,u_2v_2\in E(G)$.
There exist two consecutive vertices on the longest cycle in $H$ as the neighbors of $u_1,u_2$.
Using these, we can then find a cycle of length at least
$\lceil\frac{c}{2}\rceil+2+(c-6)\geq c+1$ (as $c\geq 8$ is odd).
This final contradiction completes the proof of Theorem \ref{Th:BondyStab}.
\end{proof}

Now we can prove Theorem \ref{Th:BondyStab-spec}.

\medskip

{\noindent \it Proof of Theorem \ref{Th:BondyStab-spec}.}
By Theorem \ref{Th:BondyStab}, one of its three cases holds.
Since $c\geq 10$, $(a)$ does not hold.
Suppose that $(c)$ holds, i.e., there exists a cycle $C'$ with $|V(C\cap C')|\leq 1$.
Since $G$ is 2-connected,
there exist two disjoint paths $P_1, P_2$ from $x_1,x_2\in V(C)$ to $y_1,y_2\in V(C')$, respectively;
moreover, in the case of $|V(C\cap C')|=1$, the path $P_2$ can be chosen
so that $P_2$ consists of the single vertex in $V(C\cap C')$.
One can then find a cycle $D$ in the union $C\cup C'\cup P_1\cup P_2$
satisfying that $|D|\geq \lceil\frac{c}{2}\rceil+\lceil|C'|/2\rceil+|P_1|+|P_2|$.
If $V(C\cap C')=\emptyset$, then $|D|\geq \lceil\frac{c}{2}\rceil+(\lfloor\frac{c}{2}\rfloor-1)+2=c+1$,
a contradiction; otherwise $|V(C\cap C')|=1$,
then $|D|\geq \lceil\frac{c}{2}\rceil+\lfloor\frac{c}{2}\rfloor+1=c+1$, also a contradiction.
This shows that $(c)$ does not hold.
Hence, $(b)$ holds, i.e., there exists a vertex $u\in V(G-C)$ with $d_C(u)=\lfloor\frac{c}{2}\rfloor$.
It remains to show that $u$ is an isolated vertex in $G-C$.
Suppose this is not the case.
Then $u$ is contained in a component $R$ of $G-C$ with $|R|\geq 2$.
Since $G$ is 2-connected, there exists a vertex $v\in V(R-u)$ with a neighbor in $V(C)$.
Using this, one can easily find a cycle of length at least $c+1$.
This finishes the proof of Theorem \ref{Th:BondyStab-spec}.
\qed

\subsection{Proof of Theorem \ref{Th:BondyStab-2con}}\label{subsec:Bondy2}

To prove Theorem \ref{Th:BondyStab-2con}, in view of Theorem \ref{Th:BondyStab-spec},
it suffices to show the following lemma.

\begin{lem}\label{Le:LonCyc}
Let $G$ be a 2-connected non-Hamiltonian graph on $n$ vertices and
$C$ be a longest cycle in $G$ of length $c$.
Suppose that there exists an isolated vertex $u$ in $G-C$ with $d_C(u)=\lfloor\frac{c}{2}\rfloor$.\\
-- If $c$ is even, then $G\subseteq W_{n,\lfloor\frac{c}{2}\rfloor,c}$.\\
-- If $c\geq 9$ is odd, then $G\subseteq W_{n,\lfloor\frac{c}{2}\rfloor,c}$ or $G$ is a subgraph of a member of $\mathcal{X}_{n,c}\cup \mathcal{Y}_{n,c}$.
\end{lem}

\begin{proof}
Throughout this proof, let $N:=N_C(u)$.

We first consider the case that $c$ is even. Let $C=x_1x_2...x_cx_1$.
We may assume that $N=\{x_1,x_3,...,x_{c-1}\}$.
Consider any component $R$ in $G-C$ with $u\notin V(R)$.
As $G$ is 2-connected, $|N_C(R)|\geq 2$.
We also have that $C-N$ consists of isolated vertices and $|N_C(R)\backslash N|\leq 1$
(otherwise one can easily find a cycle longer than $C$ using Posa's rotation technique).
Suppose that there exist some vertices say $x_2\in N_C(R)\backslash N$ and $x_i\in N_C(R)\cap N$.
We assume that $i\notin \{1,3\}$ (as otherwise there is a cycle longer than $C$).
There exists an $(x_2,R,x_i)$-path $P$ of length at least 2,
then $P\cup (C-x_{i+1}-x_2x_3)\cup \{ux_3,ux_{i+2}\}$ forms a cycle of length at least $c+1$, a contradiction.
Thus $N_C(R)\subseteq N$.
If $|R|\geq 2$, then there exist distinct $x_i, x_j\in N_C(R)\cap N$ and an $(x_i,R,x_j)$-path $Q$ of length at least 3.
One can find a longer cycle easily if the distance between $x_i$ and $x_j$ on $C$ is two;
otherwise, $Q\cup (C-\{x_{i+1},x_{j+1}\})\cup \{ux_{i+2},ux_{j+2}\}$ forms a cycle longer than $C$.
This shows that $|R|=1$ and $N_C(R)\subseteq N$ for any component $R$ in $G-C$.
Therefore indeed $G$ is a subgraph of $W_{n,\frac{c}{2},c}$ when $c$ is even.

From now on we assume that $c\geq 9$ is odd.
Let $c:=2\alpha+1$ and $C=x_1x_2\ldots x_{2\alpha+1}x_1$, where $\alpha\geq 4$.
We may assume that $N=N_C(u)=\{x_1,x_3,...,x_{2\alpha-1}\}$.
First we observe an easy fact that
$C-N$ consists of a unique edge $x_{2\alpha}x_{2\alpha+1}$ and isolated vertices.
Next we determine the structures of all components $R$ in $G-C$.

\medskip

{\noindent \bf Claim.} Any component $R$ in $G-C$ is of one of the following three types:\\
(i) $|R|=1$ and $N_C(R)\subseteq N_C(u)$;\\
(ii) $|R|=1$ and $N_C(R)=\{x_{2\alpha-1},x_{2\alpha+1}\}$ or $N_C(R)=\{x_1,x_{2\alpha}\}$;\\
(iii) $R$ is an induced star, which is $\{x_1,x_{2\alpha-1}\}$-feasible.\footnote{Recall the definition of {\it $\{a,b\}$-feasible} from Subsection 1.2.}

\medskip

{\noindent \it Proof of Claim.}
First assume that there are two vertices $a,b$ in $N_C(R)\backslash N$.
Then there exists an $(a,R,b)$-path $P$ of length at least 2.
If $\{a,b\}=\{x_{2\alpha},x_{2\alpha+1}\}$, then $(C-x_{2\alpha}x_{2\alpha+1})\cup P$ forms a cycle of length at least $c+1$.
Otherwise, we have either $a^+,b^+\in N$ or $a^-,b^-\in N$. We may assume the former case occurs.
Then $P\cup (C-\{aa^+,bb^+\})\cup \{ua^+,ub^+\}$ forms a cycle of length at least $c+1$, a contradiction.

Now assume that $N_C(R)\subseteq N$.
If $|R|=1$, then $R$ is of type (i). So $|R|\geq 2$.
As $G$ is 2-connected, there exist $x_{2i-1}, x_{2j-1}\in N_C(R)$ and an $(x_{2i-1},R,x_{2j-1})$-path $P$ of length at least 3.
Suppose that $\{x_{2i-1}, x_{2j-1}\}\neq \{x_1,x_{2\alpha-1}\}$.
If the distance between $x_{2i-1}$ and $x_{2j-1}$ on $C$ is two, then it is easy to find a cycle of length at least $c+1$;
otherwise, since $\alpha\geq 4$, without loss of generality we may assume that $1\leq 2j-1<2j+1<2i-1<2i+1\leq 2\alpha-1$,
then $P\cup (C-\{x_{2j},x_{2i}\})\cup \{ux_{2j+1},ux_{2i+1}\}$ forms a cycle of length at least $c+1$, a contradiction.
This shows that $N_C(R)=\{x_1,x_{2\alpha-1}\}$.
If there exists an $(x_1,R,x_{2\alpha-1})$-path $P$ of length at least 4,
then $(C-\{x_{2\alpha},x_{2\alpha+1}\})\cup P$ is a cycle of length at least $c+1$.
Hence, all $(x_1,R,x_{2\alpha-1})$-paths in $G[R\cup N_C(R)]$ are of length 3.
This forces $R$ to be an induced star, and moreover, if $|R|\geq 3$,
then all leaves of $R$ are only adjacent to the same vertex in $\{x_1,x_{2\alpha-1}\}$.
So $R$ is of type (iii).

It remains to consider that $|N_C(R)\backslash N|=1$.
As $G$ is 2-connected, there exists some $x_{2j-1}\in N_C(R)\cap N$.
Let $P$ be an $(x_{2j-1},R,N_C(R)\backslash N)$-path of length at least 2.
Let us first consider that $2\leq j\leq\alpha-1$.
If $N_C(R)\backslash N=\{x_{2i}\}$, where $1\leq i\leq \alpha$,
then we may assume that $x_{2j-1}$ and $x_{2i}$ are not adjacent (as otherwise there is a longer cycle).
By symmetry, we may also assume $2j-1<2i-1<2i$.
Thus we have $1\leq 2j-3<2j-1<2i-1<2i\leq 2\alpha$.
Then $P\cup (C-x_{2j-2}-x_{2i-1}x_{2i})\cup \{x_{2j-3}u,ux_{2i-1}\}$
forms a cycle of length at least $c+1$, a contradiction.
So, $N_C(R)\backslash N=\{x_{2\alpha+1}\}$.
Then, $(C-x_{2j}-x_{2\alpha+1}x_1)\cup P\cup \{x_1u,ux_{2j+1}\}$ is a cycle of length at least $c+1$, again a contradiction.
Hence, we have that $j\in \{1,\alpha\}$.
By symmetry, we may just consider $j=1$.
In this case, $x_1\in N_C(R)\cap N$ (so clearly $x_2, x_{2\alpha+1}\notin N_C(R)$) and we claim that $N_C(R)\backslash N=\{x_{2\alpha}\}$.
Suppose for a contradiction that $N_C(R)\backslash N=\{x_{2i}\}$ for $2\leq i\leq \alpha-1$.
Then $(C-x_2-x_{2i}x_{2i+1})\cup P\cup \{x_{3}u,ux_{2i+1}\}$ forms a cycle of length at least $c+1$, a contradiction.
This shows that $N_C(R)=\{x_1,x_{2\alpha}\}$. If $|R|\geq 2$, then $P$ can be chosen to be a path of length at least 3
and the cycle $P\cup (C-x_{2\alpha+1})$ contradicts the maximality of $C$.
Therefore, $|R|=1$ and $N_C(R)=\{x_1,x_{2\alpha}\}$. So $R$ is of type (ii). This proves the claim. \qed

\medskip

We show that all components $R$ in $G-C$ of type (ii)
have the same two neighbors in $C$ (say $N_C(R)=\{x_1, x_{2\alpha}\}$).
Otherwise there are two components in $G-C$ of type (ii), say $R_1=\{v_1\}$ and $R_2=\{v_2\}$,
such that $N_C(v_1)=\{x_{2\alpha-1},x_{2\alpha+1}\}$ and $N_C(v_2)=\{x_1, x_{2\alpha}\}$,
then $G[V(C)\cup \{v_1,v_2\}]$ contains a cycle of length $c+2$, a contradiction.

If all components in $G-C$ are of type (i),
then as $N^+$ is independent, we have $G\subseteq W_{n,\lfloor\frac{c}{2}\rfloor,c}$.
So there exists at least one component in $G-C$ of type (ii) or (iii).

Suppose that there is no component in $G-C$ of type (iii).
Then there exists some component in $G-C$, say $\{v\}$, of type (ii).
So we can assume $N_C(v)=\{x_1, x_{2\alpha}\}$.
We show that $N_G(x_{2\alpha+1})=\{x_1,x_{2\alpha}\}$.
To see this, consider $C':=(C-\{x_1x_{2\alpha+1},x_{2\alpha+1}x_{2\alpha}\})\cup \{x_1v,vx_{2\alpha}\}$,
which also is a longest cycle in $G$.
Then $x_{2\alpha+1}$ is contained in a component $R'$ in $G-C'$.
As $N_{C'}(R')\supseteq \{x_1,x_{2\alpha}\}$, by the Claim,
$R'$ must be of type (ii) and thus we have $N_G(x_{2\alpha+1})=\{x_1,x_{2\alpha}\}$.
Let $J_1$ (resp. $J_2$) be the set of all vertices in components in $G-C$ of type (i) (resp. type (ii)).
Now set $A:=N$, $B:=N^+\cup J_1$ and $X:=\{x_{2\alpha+1}\}\cup J_2$.
Then both $B$ and $X$ are stable and for any $w\in X$, $N_G(w)=\{x_1,x_{2\alpha}\}$.
This shows that $G$ is a subgraph of some graph from $\mathcal{X}_{n,c}$.

Now we assume that there exists some component $R$ in $G-C$ of type (iii).
Let $J_1,J_2, J_3$ be the sets of all vertices in components in $G-C$ of type (i), (ii), (iii), respectively.
Set $A:=N$, $B:=\{x_2,x_4,...,x_{2\alpha-2}\}\cup J_1$, and $Y:=\{x_{2\alpha},x_{2\alpha+1}\}\cup J_2\cup J_3$.
Clearly $B$ is stable.
Since every vertex $v\in J_2$ satisfies $N_C(v)=\{x_1, x_{2\alpha}\}$,
we see that $G[\{x_{2\alpha},x_{2\alpha+1}\}\cup J_2]$ induces a star, say $S$, with the center $x_{2\alpha}$.
If we can show that $S$ is $\{x_1,x_{2\alpha-1}\}$-feasible,
then $G$ is a subgraph of some graph from $\mathcal{Y}_{n,c}$ (note that $G[Y]$ has at least two stars).
To show this, we note that there exists an edge $xy$ in $R$ such that
$C':=(C-\{x_{2\alpha},x_{2\alpha+1}\})\cup \{x_1x,xy,yx_{2\alpha-1}\}$ is a longest cycle in $G$.
Then $S=G[\{x_{2\alpha},x_{2\alpha+1}\}\cup J_2]$ is contained in a component $R'$ in $G-C'$.
By the Claim, $R'$ must be of type (iii), i.e., $R'$ (and thus $S$) is $\{x_1,x_{2\alpha-1}\}$-feasible.
This proves Lemma \ref{Le:LonCyc}.
\end{proof}

We have completed the proof of Theorem \ref{Th:BondyStab-2con}.


\section{Stability from many edges spanned in a long cycle}
In this section, we prove the following strengthened version of Theorem \ref{Th:WZ+closure1},
where the longest cycle in Theorem \ref{Th:WZ+closure1} is generalized to a locally maximal cycle.

Recall that $h(n,k)=\binom{n-k}{2}+k(k-1)$.

\begin{thm}\label{Th:WZ+closure}
Let $G$ be a 2-connected graph on $n$ vertices with $\delta(G)\geq k$ and
$C$ be a locally maximal cycle in $G$ of length $c\in [6,n-1]$.
If
$e(G)>\max\left\{f\left(n,k+1,c\right),f\left(n,\left\lfloor\frac{c}{2}\right\rfloor-1,c\right)\right\}$
and
$e(G[C])>h(c+1,\lfloor\frac{c}{2}\rfloor-1),$
then either $G\subseteq W_{n,\lfloor\frac{c}{2}\rfloor,c}$,
or $\overline{G}\in \{W_{n,k,c}, ~Z_{n,k,c}\}$, where $\overline{G}$ is the $C$-closure of $G$.
\end{thm}

We will reduce Theorem \ref{Th:WZ+closure} to the following three lemmas,
which are needed when dealing with the two situations arising from Lemma \ref{Le:HamConStruc2}.

\begin{lem}\label{Le:c/2-1:c/2}
Let $G_c$ be a Hamiltonian graph on $c\geq 6$ vertices.
Further suppose that $G_c$ is $(c+1)$-closed and non-Hamiltonian-connected with $e(G_c)>h(c+1,\lfloor\frac{c}{2}\rfloor-1)$.
If there exist $\lfloor\frac{c}{2}\rfloor-1$ vertices of degree at most $\lfloor\frac{c}{2}\rfloor$ in $G_c$,
then $G_c= W_{c,\lfloor\frac{c}{2}\rfloor,c}$.
\end{lem}

\begin{lem}\label{lem:CliqNum}
Let $G$ be a 2-connected graph on $n$ vertices and
$C$ be a locally maximal cycle in $G$ of length $c\leq n-1$.
Suppose that $e(G)>\max\left\{f\left(n,k+1,c\right),f\left(n,\left\lfloor\frac{c}{2}\right\rfloor-1,c\right)\right\}$.
If $G[C]$ contains a subset $S$ of $s-1$ vertices of degree at most $s$ in $G[C]$
for some integer $2\leq s\leq \lfloor\frac{c}{2}\rfloor-1$ such that $G[C]-S$ is a clique,
then $2\leq s\leq k$ and the clique number of $G[C]$ is at least $c-k+1$.
\end{lem}

\begin{lem}\label{lem:CliqNum2}
Let $G$ be a 2-connected graph on $n$ vertices with $\delta(G)\geq k$ and
$C$ be a locally maximal cycle in $G$ of length $c\leq n-1$.
If the clique number of $G[C]$ is at least $c-k+1$,
then $G\in \{W_{n,k,c}, ~Z_{n,k,c}\}$.
\end{lem}

This reduction will be done in Subsection \ref{subsec:5.0}.
We then prove these lemmas in Subsections \ref{subsec:5.1}, \ref{subsec:5.2} and \ref{subsec:WZ3}, respectively.

\subsection{Reducing Theorem \ref{Th:WZ+closure} to the lemmas}\label{subsec:5.0}

{\noindent \it Proof of Theorem \ref{Th:WZ+closure}.} (Assuming Lemmas \ref{Le:c/2-1:c/2}, \ref{lem:CliqNum} and \ref{lem:CliqNum2}.)
Let $G, C$ be as in Theorem \ref{Th:WZ+closure}. Let $\overline{G}$ be the $C$-closure of $G$.
Since $G\subseteq \overline{G}$,
we see that $\overline{G}$ is 2-connected with $\delta(\overline{G})\geq k$ and $e(\overline{G})\geq e(G)$.
By Lemma \ref{Le:Coper}, we see that the cycle $C$ remains a locally maximal cycle of length $c$ in $\overline{G}$.
By Lemma \ref{Le:CoperHC}, $\overline{G}[C]$ is non-Hamiltonian-connected.
It is also clear that $\overline{G}[C]$ is $(c+1)$-closed and $e(\overline{G}[C])\geq e(G[C])>h(c+1,\lfloor\frac{c}{2}\rfloor-1)$.


Applying Lemma \ref{Le:HamConStruc2} (with $\delta=1$) to $\overline{G}[C]$, we see that
one of the following holds:
\begin{itemize}
\item[(i)] $\overline{G}[C]$ contains a subset of $\lfloor\frac{c}{2}\rfloor-1$ vertices of degree at most
$\lfloor\frac{c}{2}\rfloor$ in $\overline{G}[C]$, or
\item[(ii)]  $\overline{G}[C]$ contains a subset $S$ of $s-1$ vertices of degree at most
$s$ in $\overline{G}[C]$ for some $2\leq s\leq \lfloor\frac{c}{2}\rfloor-2$ such that
$\overline{G}[C]-S$ is a clique.
\end{itemize}
Suppose that (i) holds. Then by Lemma \ref{Le:c/2-1:c/2} (applied to $\overline{G}[C]$),
we have $\overline{G}[C]= W_{c,\lfloor\frac{c}{2}\rfloor,c}$.
Let $B\subseteq V(C)$ consist of all vertices of degree $c-1$ in $\overline{G}[C]$.
We observe that for any two vertices $x,y\in V(C)$, if $\{x,y\}\nsubseteq B$,
then there is a Hamiltonian path from $x$ to $y$ in $\overline{G}[C]$.
Thus, for any component $R$ in $\overline{G}-C$, if $N_C(R)\nsubseteq B$,
then there is a cycle $C'$ longer than $C$ with $|E(C')\cap E(C,G-C)|\leq 2$,
a contradiction. So, we have $N_C(R)\subseteq B$. Furthermore, for any $\{x,y\}\subseteq B$,
there is an $(x,y)$-path of length at least $c-2$ in $\overline{G}[C]$.
If $|V(R)|\geq 2$, as $\overline{G}$ is 2-connected,
we can find a cycle $C'$ longer than $C$ with $|E(C')\cap E(C,G-C)|\leq 2$, a contradiction.
Hence, for any component $R$ in $\overline{G}-C$, we have $|V(R)|=1$ and $N_C(R)\subseteq B$.
This implies that $G\subseteq \overline{G}\subseteq W_{n,\lfloor\frac{c}{2}\rfloor,c}$.

So we may assume that (ii) holds.
By Lemma \ref{lem:CliqNum} (applied to $\overline{G}$ and $C$),
we get that the clique number of $\overline{G}[C]$ is at least $c-k+1$.
By Lemma \ref{lem:CliqNum2}, this shows that $\overline{G}\in \{W_{n,k,c}, ~Z_{n,k,c}\}$,
completing the proof of Theorem \ref{Th:WZ+closure}.
\qed

\subsection{Proof of Lemma \ref{Le:c/2-1:c/2}}\label{subsec:5.1}


\begin{proof}
Throughout this proof, define
$\alpha:=\left\lfloor\frac{c}{2}\right\rfloor, ~e:=e(G_c), ~A:=\{u\in V(G_c): d(u)\leq \alpha\}, \text{ and } B:=V(G_c)\backslash A.$
Since $G_c$ is $(c+1)$-closed, $B$ induces a clique.
Let $V(G_c)=\{u_1,u_2,\ldots,u_c\}$ and $f_1\leq f_2\leq \ldots\leq f_c$ be the degree sequence of $G_c$
such that $d(u_i)=f_i$ for every $1\leq i\leq c$.
There are $\alpha-1$ vertices of degree at most $\alpha$ in $G_c$, in other words,
we have $f_{\alpha-1}\leq \alpha.$

We establish some facts to be used later.
The first two facts are straightforward.

\begin{fact}
If $G_c$ has $t$ vertices of degree at most $r$, then $e(G_c)\leq tr+ \binom{c-t}{2}$.
\end{fact}


\begin{fact}
We have
\begin{align*}
h\left(1+c,\left\lfloor\frac{c}{2}\right\rfloor-1\right)
&=\frac{1}{2}{\left\lceil\frac{c}{2}\right\rceil}^2+\frac{3}{2}\left\lceil\frac{c}{2}\right\rceil
+{\left\lfloor\frac{c}{2}\right\rfloor}^2-3\left\lfloor\frac{c}{2}\right\rfloor+3
=\left\{\begin{array}{ll}
      \frac{3c^2}{8}-\frac{3c}{4}+3 &\text{ if } c \text{ is even}\\
       \frac{3\alpha^2}{2}-\frac{\alpha}{2}+5 &\text{ if } c \text{ is odd}.
    \end{array}\right.
\end{align*}
\end{fact}

\begin{fact}
$f_{\alpha-1}=\alpha$ and $f_{c-\alpha}\leq c-\alpha$. Thus, when $c$ is even, we have $f_{\frac{c}{2}}=\frac{c}{2}$.
\end{fact}

\begin{proof}
Suppose for a contradiction that $f_{\alpha-1}\leq \alpha-1$.
By Facts 1 and 2, we have the following: if $c$ is even, then
$e\leq ({\frac{c}{2}}-1)^2+\binom{\frac{c}{2}+1}{2}=\frac{3c^2}{8}-\frac{3c}{4}+1<h(1+c,\lfloor\frac{c}{2}\rfloor-1)<e$,
a contradiction; if $c$ is odd, then
$e\leq ({\lfloor\frac{c}{2}\rfloor}-1)^2+\binom{\lceil\frac{c}{2}\rceil+1}{2}=\frac{3\alpha^2}{2}-\frac{\alpha}{2}+2<h(1+c,\lfloor\frac{c}{2}\rfloor-1)<e$,
also a contradiction. Thus, $f_{\alpha-1}=\alpha$.

Suppose that $f_{c-\alpha}\geq c-\alpha+1$.
First assume that $c$ is even. As $f_{\alpha}=f_{c-\alpha}\geq \alpha+1$,
we have $d(u_{\alpha})+d(u_{\alpha-1})\geq c+1$,
so $u_{\alpha-1}$ is adjacent to all vertices in $\{u_{\alpha},u_{\alpha+1},\ldots,u_c\}$.
This implies that $d(u_{\alpha-1})=f_{\alpha-1}\geq \alpha+1$, a contradiction.
Now consider that $c$ is odd. As $f_{\alpha+1}\geq \alpha+2$, $u_{\alpha-1}$
is adjacent to all vertices in $\{u_{\alpha+1},u_{\alpha+2},\ldots,u_{2\alpha+1}\}$,
thus $f_{\alpha-1}=d(u_{\alpha-1})\geq \alpha+1$, again a contradiction. This finishes the proof.
\end{proof}

\begin{fact}
For every vertex $u\in V(G_c)$, either $d(u)=c-1$ or $2\leq d(u)\leq c-3$.
\end{fact}

\begin{proof}
Suppose for a contradiction that there exists a vertex $u$ with $d(u)=c-2$.
Then there exists a vertex $v$ not adjacent to $u$.
As $G_c$ is $(c+1)$-closed, we have $d(u)+d(v)\leq c$, implying that $d(v)\leq 2$.
Since $G_c$ is Hamiltonian, $d(v)=2$. When $c$ is even, we have
$e\leq 2+\sum_{j=2}^{\frac{c}{2}}f_j+e(G[\{u_{\frac{c}{2}+1},\ldots,u_c\}])\leq 2+(\frac{c}{2}-1)\frac{c}{2}+\binom{\frac{c}{2}}{2}
=\frac{3c^2}{8}-\frac{3c}{4}+2<h(1+c,\lfloor\frac{c}{2}\rfloor-1)<e$, a contradiction.
If $c$ is odd, then $e\leq 2+\sum_{j=2}^{\alpha+1}f_j+e(G[\{u_{\alpha+2},\ldots,u_c\}])
\leq 2+(\alpha-2)\alpha+2(\alpha+1)+\binom{\alpha}{2}
=\frac{3\alpha^2}{2}-\frac{\alpha}{2}+4<e$, a contradiction. This proves Fact 4.
\end{proof}

\begin{fact}
If it exists, let $u\in B$ be the vertex such that $d(u):=c-i$ is maximum over all vertices in $B$
with degree at most $c-2$, where $3\leq i\leq c-\alpha-1$.
If $|B|\geq i$, then $e(G_c)\leq i^2-\frac{i}{2}(|B|+\alpha+2)+\frac{\alpha}{2}(c+1)+\frac{|B|}{2}(c-\alpha)$.
\end{fact}

\begin{proof}
Since $B$ induces a clique,
all $i-1$ non-neighbors of $u$ are in $A$.
Let $A'$ be the subset of $A$ consisting of such $i-1$ vertices.
Since $G_c$ is $(c+1)$-closed, every vertex $x\in A'$ has $2\leq d(x)\leq i$.
Choose a fixed vertex $x\in A'$ and let $B'\subseteq B$ be the set of all non-neighbors of $x$ in $B$.
Then $|B'|\geq |B|-i$, and for any $y\in B'$, we have $d(y)\leq c-d(x)\leq c-2$.
By Fact 4, we see that any $y\in B'$ has degree at most $c-3$,
therefore, by the choice of $u$, $d(y)\leq d(u)=c-i$.
Now we get that
\begin{align*}
e(G_c)&=\frac{1}{2}\sum_{i=1}^{c}f_i=\frac{1}{2}(\sum_{v\in A'}d(v)+\sum_{v\in A\backslash A'}d(v)+\sum_{v\in B'}d(v)+\sum_{v\in B\backslash B'}d(v))\\
&\leq \frac{1}{2}\left((i-1)i+(|A|-i+1)\alpha+|B'|(c-i)+(|B|-|B'|)(c-1)\right)\\
&\leq i^2-i(|B|+\alpha+2)/2+\alpha(c+1)/2+|B|(c-\alpha)/2,
\end{align*}
where the last inequality holds because $|A|=c-|B|$ and $|B'|\geq |B|-i$.
\end{proof}

We divide the rest of the proof into two cases depending on the parity of $c$.

\begin{case}
$c$ is even.
\end{case}

In this case, we have $f_{\frac{c}{2}-1}=f_{\frac{c}{2}}=\frac{c}{2}$.
First we claim that $f_{\frac{c}{2}+2}\geq \frac{c}{2}+1$.
Otherwise, $f_{\frac{c}{2}+2}=\frac{c}{2}$, then $|A|\geq \frac{c}{2}+2$ and $|B|\leq\frac{c}{2}-2$.
This implies that $e=\frac{1}{2}\sum_{j=1}^{c}f_j\leq \frac{1}{2}((\frac{c}{2}+2)\frac{c}{2}+(\frac{c}{2}-2)(c-1))=\frac{3c^2}{8}-\frac{3c}{4}+1<h(1+c,\lfloor\frac{c}{2}\rfloor-1)<e$,
a contradiction.

Next we show that $f_{\frac{c}{2}+1}\geq \frac{c}{2}+1$.
Suppose not. Then we have $f_{\frac{c}{2}+1}=\frac{c}{2}$ and $|B|=\frac{c}{2}-1$.
Suppose there exists some vertex in $B$ with degree at most $c-2$.
By Fact 5, there exists some $3\leq i\leq \frac{c}{2}-1$ such that
$e\leq i^2-\frac{i}{2}(c+1)+\frac{3c^2}{8}
\leq \frac{3c^2}{8}-\frac{3c}{4}+3<e$, a contradiction.
To see why the second inequality holds, let $f(i):=i^2-\frac{i}{2}(c+1)$;
then we have $f(i)\leq \max\{f(3),f(\frac{c}{2}-1)\}$
and it is routine to check that as $c\geq 6$, this is at most $-\frac{3c}{4}+\frac{3}{2}$.
Hence we may assume that every vertex in $B$ has degree $c-1$.
Let $H$ be the spanning subgraph of $G_c$ consisting of all edges in $E(B)\cup (A,B)$.
As $e(H)=(\frac{c}{2}+1)(\frac{c}{2}-1)+\binom{\frac{c}{2}-1}{2}=\frac{3c^2}{8}-\frac{3c}{4}$
and $e>\frac{3c^2}{8}-\frac{3c}{4}+3$, we see that $E(A)$ has at least 4 edges.
Observe that every vertex in $A$ has degree at most $c/2$ in $G_c$ and is already adjacent to the $c/2-1$ vertices in $B$.
This shows that there exists a matching of size at least 4 in $A$.
One can check that the subgraph obtained from $H$ by adding a matching
of size 3 in $A$ is already Hamiltonian-connected.
So is the host graph $G_c$. But this is a contradiction.
This proves that $f_{\frac{c}{2}+1}\geq \frac{c}{2}+1$ and thus $|A|=|B|=\frac{c}{2}$.

Lastly, we show that any vertex $u\in B$ has degree $c-1$.
Suppose for a contradiction that there exists a vertex $u\in V(B)$ with $d(u)\leq c-2$.
By Fact 5, there exists some $3\leq i\leq \frac{c}{2}-1$ such that
$e\leq i^2-\frac{i}{2}(c+2)+\frac{3c^2}{8}+\frac{c}{4}\leq \frac{3c^2}{8}-\frac{3c}{4}+3<e$,
where the second inequality can be verified similarly as above for $c\geq 6$, a contradiction.

Now, we see that $B$ induces a clique $K_{\frac{c}{2}}$ and $(A,B)$ is complete bipartite.
As every vertex in $A$ has degree at most $\frac{c}{2}$, we see that $E(A)$ contains no edge.
This shows $G_c=W_{c,\frac{c}{2},c}$, completing the proof of Case 1.

\begin{case}
$c$ is odd.
\end{case}

Let $H$ be the spanning subgraph of $G_c$ consisting of all edges in $E(B)\cup (A,B)$.
In this case, $c=2\alpha+1$, where $\alpha\geq 3$.
By Facts 1 and 2, we have
$\frac{3\alpha^2}{2}-\frac{\alpha}{2}+5<e\leq (\alpha-1)\alpha+\binom{\alpha+2}{2}=\frac{3\alpha^2}{2}+\frac{\alpha}{2}+1$;
by Fact 3, $f_{\alpha-1}=\alpha$ and $f_{\alpha+1}\leq \alpha+1.$

We show in a sequence of claims that $f_\alpha =f_{\alpha+1}=\alpha+1$.
First we show $f_{\alpha+3}\geq \alpha+1$.
Otherwise, $f_{\alpha+3}\leq\alpha$. Then $|A|\geq \alpha+3$ and $|B|\leq \alpha-2$,
from which we derive a contradiction that $e\leq \frac{1}{2}((\alpha+3)\alpha)+(\alpha-2)(2\alpha))=\frac{3\alpha^2-\alpha}{2}<e.$

Next we show that $f_{\alpha+2}\geq \alpha+1$.
Suppose not. Then $f_{\alpha-1}=f_{\alpha}=f_{\alpha+1}=f_{\alpha+2}=\alpha$.
So $B=\{u_{\alpha+3},\ldots,u_{2\alpha+1}\}$ and $|B|=\alpha-1$.
Suppose that there are vertices in $B$ with degree at most $c-2$.
Let $u\in B$ be such a vertex with maximum degree $d(u)=c-i$, where $3\leq i\leq \alpha$.
If $3\leq i\leq \alpha-1$, then by Fact 5, we have
$e\leq i^2-\frac{i}{2}(2\alpha+1)+(\alpha^2+\alpha)+\frac{\alpha^2-1}{2}\leq \frac{3\alpha^2}{2}-\frac{\alpha}{2}+5<e$,
where the second inequality holds since $i^2-\frac{i}{2}(2\alpha+1)$ takes the maximum at $i=3$ or $\alpha-1$.
This is a contradiction.
So $i=\alpha$, that is, $d(u)=\alpha+1$. Then $e\leq \frac12((\alpha+2)\alpha+(\alpha+1)+(\alpha-2)(2\alpha))=\frac12(3\alpha^2-\alpha+1)
<\frac{3\alpha^2}{2}-\frac{\alpha}{2}+5<e$, again a contradiction.
Now we may assume that every vertex in $B$ has degree $2\alpha$.
So $(A,B)$ is complete bipartite and thus every vertex in $A$ has degree $\alpha-1$ in the subgraph $H$ defined above.
By the definition of $A$, every vertex in $A$ has degree at most $\alpha$ in $G_c$.
This shows that $E(A)$ must be a matching (if not empty).
Since $e(H)=\binom{\alpha-1}{2}+(\alpha-1)(\alpha+2)=\frac{3\alpha^2}{2}-\frac{\alpha}{2}-1$ and $e>\frac{3\alpha^2}{2}-\frac{\alpha}{2}+5$,
we see that $E(A)$ forms a matching of size at least 7.
One can check that the subgraph obtained from $H$ by adding a matching of size 4 in $A$ is Hamiltonian-connected,
so the host graph $G_c$ is also Hamiltonian-connected, a contradiction. This proves $f_{\alpha+2}\geq \alpha+1$.

We also claim that $f_{\alpha+1}=\alpha+1$.
Suppose not. Then we have $f_{\alpha-1}=f_{\alpha}=f_{\alpha+1}=\alpha$ and $B=\{u_{\alpha+2},\ldots,u_{2\alpha+1}\}$.
So $|B|=\alpha$.
First suppose that every vertex $b\in B$ has degree $c-1$.
Then the subgraph $H$ is just a vertex-disjoint union of a clique $K_{\alpha}$ and an independent set of size $\alpha+1$,
with a complete bipartite subgraph between the two parts.
So $\frac{3\alpha^2}{2}+\frac{\alpha}{2}+1\geq e\geq e(H)=\alpha(\alpha+1)+\binom{\alpha}{2}=\frac{3\alpha^2}{2}+\frac{\alpha}{2}$,
which implies that $E(A)$ has at most one edge.
Thus, $G_c=H$ or $W_{c,\lfloor\frac{c}{2}\rfloor,c}$. But for the former case, $G_c=H$ is not Hamiltonian.
Hence in this case, we prove $G_c=W_{c,\lfloor\frac{c}{2}\rfloor,c}$.
Now we may assume that there are vertices in $B$ of degree at most $c-2$.
Let $u\in B$ be such a vertex with maximum degree $d(u)=c-i$, where $3\leq i\leq \alpha$.
If $3\leq i\leq \alpha-1$, by Fact 5 we have
$e\leq i^2-i(\alpha+1)+\frac{3}{2}(\alpha^2+\alpha)\leq \frac{3\alpha^2}{2}-\frac{\alpha}{2}+5<e$,
where the second inequality holds since $i^2-i(\alpha+1)$ takes the maximum at $i=3$ or $\alpha-1$.
So we must have $i=\alpha$. This shows that for any $b\in B$, either $d(b)=2\alpha$ or $d(b)=\alpha+1$.
If there exist at least two vertices in $B$ of degree $\alpha+1$, then $e\leq \frac{1}{2}((\alpha+1)\alpha+2(\alpha+1)+(\alpha-2)(2\alpha))=\frac{3\alpha^2}{2}-\frac{\alpha}{2}+1<e$,
a contradiction. So $B$ contains $\alpha-1$ vertices of degree $2\alpha$ and a vertex say $u$ of degree $\alpha+1$.
Every $x\in A$ has at least $\alpha-1$ neighbors in $B$; this shows that $E(A)$ is a matching.
Note that the vertex $u$ has two neighbors in $A$.
So $e(H)=\binom{\alpha}{2}+2+(\alpha-1)(\alpha+1)=\frac{3\alpha^2}{2}-\frac{\alpha}{2}+1$.
This, together with $e>\frac{3\alpha^2}{2}-\frac{\alpha}{2}+5$, shows that
$E(A)$ is a matching of size at least 5.
One can check that $H$ plus one additional edge in $A$ is already Hamiltonian-connected.
Therefore, $G_c$ is Hamiltonian-connected, a contradiction.
This proves $f_{\alpha+1}=\alpha+1$.

We now claim that $f_\alpha=\alpha+1$.
Suppose not. Then $f_\alpha=\alpha$. As $f_{\alpha+1}=\alpha+1$, it follows that $B=\{u_{\alpha+1},\ldots,u_{2\alpha+1}\}$.
So $|A|=\alpha$ and $|B|=\alpha+1$.
Since $d(u_{\alpha+1})=\alpha+1$ and $B$ is a clique, $u_{\alpha+1}$ has only one neighbor in $A$,
say $x$. If every vertex in $B\backslash \{u_{\alpha+1}\}$ has degree $2\alpha$,
then $d(x)\geq |B|=\alpha+1$, contradicting the fact that $x\in A$.
Thus, there exist some vertices in $B\backslash \{u_{\alpha+1}\}$ of degree at most $2\alpha-1$.
Among all such vertices, choose $u\in B\backslash \{u_{\alpha+1}\}$ such that $d(u)=2\alpha-i$ is maximum.
By Fact 4, we have $2\leq i\leq \alpha-1$.
Suppose that $2\leq i\leq \alpha-2$.
By the similar argument as in Fact 5, there exists $A'=A\backslash N(u)$
with $|A'|=i$ such that $d(x)\leq i+1$ for any $x\in A'$;
and there also exists $B'\subseteq B$ with $|B'|=\alpha-i$
such that $d(y)\leq 2\alpha-i$ for any $y\in B'$ (except the vertex $u_{\alpha+1}$).
Notice that in this case, $|B'|=\alpha-i\geq 2$.
So $e\leq \frac{1}{2}(i(i+1)+(\alpha-i)\alpha+(\alpha+1)+(\alpha-i-1)(2\alpha-i)+(i+1)(2\alpha))
=i^2-(\alpha-1)i+\frac{3\alpha^2+\alpha+1}{2}\leq \frac{3\alpha^2}{2}-\frac{\alpha}{2}+5<e$,
where the second inequality holds as the maximum of $i^2-(\alpha-1)i$ occurs at either $i=2$ or $i=\alpha-2$.
This shows that for any $u\in B\backslash \{u_{\alpha+1}\}$, $d(u)=\alpha+1$ or $2\alpha$.
If there are two vertices in $B\backslash \{u_{\alpha+1}\}$ of degree $\alpha+1$,
then $e\leq \frac{1}{2}(\alpha^2+3(\alpha+1)+(\alpha-2)(2\alpha))=\frac{1}{2}(3\alpha^2-\alpha+3)<e$, a contradiction.
Hence, there exists only one vertex $u\in B\backslash \{u_{\alpha+1}\}$ with $d(u)=\alpha+1$.
Then each of $\{u_{\alpha+1},u\}$ has a neighbor in $A$, say $x, x'$, respectively.
We see that $x,x'$ are distinct (as otherwise $x=x'$ is adjacent to all vertices in $B$ and then $d(x)\geq \alpha+1$).
It is easy to see that $x,x'$ have degree $\alpha$ in $H$, while all other vertices in $A$ have degree $\alpha-1$ in $H$.
So $E(A)$ is a matching (if not empty).
Since $e(H)=\binom{\alpha+1}{2}+2+(\alpha-1)\alpha=\frac{3\alpha^2}{2}-\frac{\alpha}{2}+2$
and $e>\frac{3\alpha^2}{2}-\frac{\alpha}{2}+5$, $E(A)$ is a matching of size at least 4.
We can verify that $H$ plus any edge in $A$ (which is independent of $x,x'$) is Hamiltonian-connected;
so $G_c$ is Hamiltonian-connected as well, a contradiction. This proves that $f_\alpha=\alpha+1$.
Note that $|A|=\alpha-1$ and $|B|=\alpha+2$.

Lastly we claim that any vertex in $B\backslash \{u_{\alpha},u_{\alpha+1}\}$ has degree $2\alpha$ in $G_c$.
Suppose this is not true. Then there exists a vertex $u\in B\backslash \{u_{\alpha},u_{\alpha+1}\}$ with $d(u)=2\alpha-i$,
where $2\leq i\leq \alpha-1$, and subject to this, we choose $d(u)$ to be maximum.
Similarly as above, there is a subset $A'=A\backslash N(u)$ with $|A'|=i$ such that each vertex in $A'$ is of degree at most $i+1$.
Take any $x\in A'$. Then, there exists $B'\subseteq B\backslash N(x)$ with $|B'|=|B|-(i+1)=\alpha-i+1$
such that for any $y\in B'$, $d(y)\leq 2\alpha-i$.
If $2\leq i\leq\alpha-3$, then $|B'|=\alpha-i+1\geq 2$ and possibly $u_{\alpha},u_{\alpha+1}$ are in $B'$,
thus we have
$e\leq \frac{1}{2}(i(i+1)+(\alpha-1-i)\alpha+2(\alpha+1)+(\alpha-i-1)(2\alpha-i)+(i+1)(2\alpha))
=i^2-(\alpha-1)i+\frac{3\alpha^2+\alpha+2}{2}\leq \frac{3\alpha^2}{2}-\frac{\alpha}{2}+5<e$, a contradiction.
If $i=\alpha-2$, then $d(u)=\alpha+2$, so
$e\leq \frac{1}{2}(i(i+1)+(\alpha-1-i)\alpha+2(\alpha+1)+(\alpha+2)+(\alpha-1)(2\alpha))
=\frac{3\alpha^2}{2}-\frac{\alpha}{2}+3<e$, also a contradiction.
Hence, $i$ must be $\alpha-1$.
This implies that every vertex in $B\backslash \{u_{\alpha},u_{\alpha+1}\}$ has degree $2\alpha$ or $\alpha+1$.
If there are at least two vertices in $B\backslash \{u_{\alpha},u_{\alpha+1}\}$ of degree $\alpha+1$,
then $e\leq \frac{1}{2}((\alpha-1)\alpha+4(\alpha+1)+(\alpha-2)(2\alpha))=\frac{3\alpha^2}{2}-\frac{\alpha}{2}+2<e$.
So $B$ has exactly three vertices ($u_{\alpha},u_{\alpha+1}$ and say $u$) of degree $\alpha+1$,
while other vertices in $B$ have degree $2\alpha$.
Note that $(A, B-\{u_{\alpha},u_{\alpha+1}, u\})$ forms a complete bipartite $K_{\alpha-1,\alpha-1}$.
This, together with the fact that $B$ induces a clique, shows that $G_c$ is Hamiltonian-connected,
finishing the proof of this claim.

Now we see that $B$ induces a clique $K_{\alpha+2}$, $u_{\alpha}$ and $u_{\alpha+1}$ have no neighbors in $A$, and $(A,B-\{u_{\alpha},u_{\alpha+1}\})$ is complete bipartite.
So every vertex in $A$ has $\alpha$ neighbors in $B$, which in turn shows that $A$ is stable.
We have proved that $G_c=W_{c,\lfloor\frac{c}{2}\rfloor,c}$.
The proof of Lemma \ref{Le:c/2-1:c/2} is completed.
\end{proof}

\subsection{Proof of Lemma \ref{lem:CliqNum}: an estimate of the clique number}\label{subsec:5.2}
\noindent {\bf Lemma \ref{lem:CliqNum}.}
{\em
Let $G$ be a 2-connected graph on $n$ vertices and
$C$ be a locally maximal cycle in $G$ of length $c\leq n-1$.
Suppose that $e(G)>\max\left\{f\left(n,k+1,c\right),f\left(n,\left\lfloor\frac{c}{2}\right\rfloor-1,c\right)\right\}$.
If $G[C]$ contains a subset $S$ of $s-1$ vertices of degree at most $s$ in $G[C]$
for some integer $2\leq s\leq \lfloor\frac{c}{2}\rfloor-1$ such that $G[C]-S$ is a clique,
then $2\leq s\leq k$ and the clique number of $G[C]$ is at least $c-k+1$.
}

\medskip

To prove this, we will need some ingredient in the proof of \cite{FLW04} by Fan, Lv and Wang.
An important tool in \cite{FLW04} is an edge-switching technique, which we introduce as follows.
Let $xy$ be an edge in a graph $G$ and let $A\subset N(y)\backslash (N(x)\cup \{x\})$.
The \emph{edge-switching graph} of $G$ with respect to $A$ (from $y$ to $x$),
denoted by $G[y\rightarrow x; A]$, is the graph obtained from $G$
by deleting all the edges $yz$, $z\in A$ and adding all the edges $xz$,
$z\in A$.

\begin{lem}[Lemma 2.4, \cite{FLW04}]\label{Le:LMC-oper}
Let $G$ be a 2-connected graph, $C$ a locally maximal cycle in $G$, and
$R$ a component in $G-C$. Then one of the following holds:\\
(i) $N_R(x)=V(R)$ for every vertex $x\in N_C(R)$;\\
(ii) There exists a vertex $y\in N_R(x)$ for some $x\in N_C(R)$
and a nonempty set $A\subseteq N_R(y)\backslash (N_R(x)\cup \{x\})$ such that
\begin{align*}
G'&=\left\{\begin{array}{ll}
      G[y\rightarrow x; A] &\quad{\text{ if } G[y\rightarrow x; A] \text{ is 2-connected,}}\\
      G[y\rightarrow x; A]+yx' &\quad{\text{ otherwise}}.
    \end{array}\right.
\end{align*}
is 2-connected, where $x'\in N_C(R)\backslash \{x\}$, and $C$ remains a locally maximal cycle
in $G'$.
\end{lem}

We now prove Lemma \ref{lem:CliqNum}.
We point out that the graph $G'$ defined in Lemma \ref{Le:LMC-oper}(ii) satisfies that $e(G')\geq e(G)$ and $G'[C]=G[C]$.

\begin{proof}
When applying Lemma \ref{Le:LMC-oper}(ii),
we see that the cycle $C$ remains locally maximal in the resulting graph, which is 2-connected.
So we may repeatedly apply Lemma \ref{Le:LMC-oper}(ii).
Note that as the set $A$ is nonempty, each time Lemma \ref{Le:LMC-oper}(ii) is applied,
the number of edges not incident with $C$ strictly decreases.
So this process will eventually stop (at some graph say $G^*$);
and when it stops, (i) must occur for any component $R$ in $G^*-C$, i.e.,
\begin{align}\label{equ:RtoC}
\text{ every vertex } x\in N_{G^*}(R)\cap V(C) \text{ is adjacent to all vertices in } V(R) \text{ in } G^*.
\end{align}
Let $\omega$ be the clique number of $G[C]=G^*[C]$.
Then $\omega\geq |V(C)\backslash S|\geq c-s+1.$
Also we have that $e(G)\leq e(G^*)$ and $e(G^*[C])=e(G[C])\leq \binom{c-s+1}{2}+(s-1)s$.

Let $R_1,R_2,\ldots,R_t$ be all components of $G^*-C$.
For any $1\leq i\leq t$, let $p_i=|N_{G^{*}}(R_i)\cap V(C)|$,
and $d_i$ be the length of the longest path between any two vertices in $N_{G^{*}}(R_i)\cap V(C)$
with all internal vertices in $R_i$.
In view of \eqref{equ:RtoC}, we see that $d_i-2$ denotes the length of the longest path in $R_i$.
By a theorem of Erd\H{o}s and Gallai (see \cite[Theorem~2.6]{EG59}),
we have $e_{G^*}(R_i)\leq \frac{(d_i-2)|V(R_i)|}{2}$.
Let $R_\alpha$ be the component in $G^*-C$ which attains the maximum of $\{d_i+2p_i:1\leq i\leq t\}$,
and let $p:=p_\alpha$ and $d:=d_\alpha$.
Then
\begin{align*}
e(G)&\leq e(G^{*})\leq e(G^{*}[C])+\sum_{i}\left(\frac{(d_i-2)}{2}|V(R_i)|+p_i\cdot |V(R_i)|\right)\\
&\leq \binom{c-s+1}{2}+(s-1)s+\frac{d+2p-2}{2}(n-c).
\end{align*}

Next we claim that $d+2p\leq 2+2s$.
Suppose that $d+2p\geq 3+2s$. Consider the component $R:=R_\alpha$ in $G^*-C$.
If $d=2$, then it follows $p\geq s+1$.
Since $(N_{G^{*}}(R)\cap V(C))^+$ is an independent set in $V(C)$ of size $p$
(otherwise, it would contradict that $C$ is locally maximal in $G^*$),
we have $\omega\leq c-(p-1)\leq c-s$,
a contradiction to that $\omega\geq c-s+1$.
Now we may assume $d\geq 3$. This shows that $|V(R)|\geq 2$.
Since $G^*$ is 2-connected, by \eqref{equ:RtoC}, we see that
$N_{G^{*}}(R)\cap V(C)$ is a strong attachment of $R$ to $C$.
By Lemma \ref{Le:clique}(i), $\omega\leq c-(d-1)(p-1)$.
As $p\geq 2$, we have $(\frac{d-1}{2}-1)((p-1)-1)\geq 0$,
which implies that $\frac{d-1}{2}(p-1)\geq \frac{d-1}{2}+(p-1)-1$,
that is, $(d-1)(p-1)\geq d+2p-5$.
So $\omega\leq c-(d-1)(p-1)\leq c-(d+2p)+5\leq c-2s+2\leq c-s$, again a contradiction.
This proves the claim.

Combining the above bounds, we obtain that
$$e(G)\leq \binom{c-s+1}{2}+(s-1)s+\frac{d+2p-2}{2}(n-c)\leq\binom{c-s+1}{2}+s(n-c+s-1)=f(n,s,c).$$
If $k+1\leq s\leq \lfloor\frac{c}{2}\rfloor-1$,
then by the monotonicity of the function $f(n,k,c)$,
it holds that $e(G)\leq \max\{f(n,k+1,c),f(n,\lfloor\frac{c}{2}\rfloor-1,c\}$,
a contradiction.
Thus we must have $2\leq s\leq k$ and then $\omega\geq c-s+1\geq c-k+1,$
finishing the proof of Lemma \ref{lem:CliqNum}.
\end{proof}

\subsection{Proof of Lemma \ref{lem:CliqNum2}}\label{subsec:WZ3}

\noindent {\bf Lemma \ref{lem:CliqNum2}.}
{\em
Let $G$ be a 2-connected graph on $n$ vertices with $\delta(G)\geq k$ and
$C$ be a locally maximal cycle in $G$ of length $c\leq n-1$.
If the clique number of $G[C]$ is at least $c-k+1$,
then $G\in \{W_{n,k,c}, ~Z_{n,k,c}\}$.
}

\begin{proof} Consider any component $R$ in $G-C$.
Let $T$ be a maximum strong attachment of $R$ to $C$ and $Q:=N_C(R)\backslash T$.
Let $t:=|T|$, $q:=|Q|$ and $\omega$ be the clique number of $G[C]$. So $\omega\geq c-k+1$.
We define the triple $ch(R):=(t, q,\omega)$ to be the {\it character} of the component $R$;
and we say a component $R$ is {\it infeasible}, if $|N_C(R)|\leq k-1$ and $ch(R)\neq(2,0,c-k+1)$.

We now proceed by establishing a sequence of claims.
An important step for our proof is to show that in fact there is no infeasible component $R$ in $G-C$.

\setcounter{claim}{0}

\begin{claim}
For any component $R$ in $G-C$, both $(N_C(R))^+$ and $(N_C(R))^-$ are stable, and $|N_C(R)|\leq k$.
\end{claim}

\begin{proof}
If $(N_C(R))^+$ contains an edge say $x^+y^+$, where $x,y\in N_C(R)$,
then there exists an $(x,R,y)$-path $P$
and $C':=(C-\{xx^+,yy^+\})\cup P\cup \{x^+y^+\}$ is a longer cycle than $C$ with $|E(C')\cap E(C,G-C)|=2$, a contradiction.
So $(N_C(R))^+$ and $(N_C(R))^-$ are stable.
This implies that $c-k+1\leq \omega\leq c-|N_C(R)|+1$, proving the claim.
\end{proof}

\begin{claim}
For any infeasible component $R$ in $G-C$ with $ch(R)= (t, q,\omega)$, we have $|V(R)|\geq 2$ and $t\geq 2$.
\end{claim}
\begin{proof}
Suppose that $|V(R)|=1$, say $V(R)=\{x\}$.
By Claim 1, we have $|N_C(x)|\leq k$.
But $\delta(G)\geq k$. This shows that $|N_{C}(x)|=k$, a contradiction to the definition of an infeasible component.
So $|V(R)|\geq 2$. As $G$ is 2-connected, we have at least two independent edges between $C$ and $R$,
implying that $t\geq 2$.
\end{proof}

\begin{claim}
For any infeasible component $R$ in $G-C$, $2\leq |N_C(R)|\leq k-2$.
\end{claim}
\begin{proof}
Suppose not. Set $N:=N_C(R)$, then $|N|= k-1$.
As $G$ is 2-connected, $|N|\geq 2$, implying that $k\geq 3$.

Suppose that $\omega\geq c-k+2$.
Let $W$ be a maximum clique of size $\omega\geq c-k+2$ in $G[C]$ and $I:=V(C)\backslash W$.
By Claim 1, $N^+$ is stable, so $|W\cap N^+|\leq 1$.
By the inclusion-exclusion principle,
$c-k+2\leq |W|=|W\cup N^+|+|W\cap N^+|-|N^+|\leq c+1-(k-1)$.
This shows that $|W|=c-k+2$, $|W\cap N^+|=1$, and $W\cup N^+=V(C)$,
the last of which implies that $I\subseteq N^+$.
Similarly, we have $I\subseteq N^-$.
Then $I^+\cup I^-\subseteq N$.
So $k-1=|N|\geq |I^+\cup I^-|=|I^+|+|I^-|-|I^+\cap I^-|=2(k-2)-|I^+\cap I^-|$,
implying that $|I^+\cap I^-|\geq k-3$.
Let $C=x_1x_2...x_cx_1$.
Since $|I|=k-2$, it is not hard to see that
$I=\{x_{i+1},x_{i+3},...,x_{i+2k-5}\}$ for some $i$.
So $N=I^+\cup I^-=\{x_i,x_{i+2},...,x_{i+2k-4}\}$.
In this case, for any two $x_j,x_{j+2}\in N$,
every $(x_j,R,x_{j+2})$-path must be of length 2,
implying that $N_R(x_j)=N_R(x_{j+2})=\{x\}$ for some $x\in V(R)$.
So $x$ is the unique neighbor of $N_C(R)$ in $R$.
Since $\delta(G)\geq k$ and $|N_C(R)|=k-1$, $x$ should have other neighbors in $R$ and thus $R-\{x\}\neq \emptyset$.
But we also have $N_C(R-\{x\})=\{x\}$, contradicting that $G$ is 2-connected.

Now we may assume that $\omega=c-k+1$.
Recall the definitions of $T, Q, t, q$, respectively.
We have $t+q=|N_{C}(R)|=k-1$.
Since $|V(R)|\geq 2$, the longest $(x,R,y)$-path for all $x,y\in T$ is of length at least 3.
By Lemma \ref{Le:clique}(ii), we have $\omega\leq c-2(t-1)-q=c-k-t+3.$
If $t\geq 3$, then $\omega\leq c-k$, a contradiction.
So $t=2$. If $q=0$, then $ch(R)=(t,q,\omega)=(2,0,c-k+1)$, a contradiction.
So we may assume that $t=2$ and $q\geq 1$.
Let $T=\{x_1,x_2\}$ and $x_1y_1,x_2y_2$ be two independent edges in $E(R,C)$, where $y_1,y_2\in V(R)$.
Suppose that $|V(R)|=2$. Then $V(R)=\{y_1,y_2\}$.
Since $\delta(G)\geq k$, we have $d_C(y_1)\geq k-1$ and $d_C(y_2)\geq k-1$.
So $N=N_C(y_1)=N_C(y_2)$ and every vertex in $N$ belongs to $T$. So $q=0$, a contradiction.

It remains to consider $|V(R)|\geq 3$.
As $q\geq 1$, there exists some vertex $w\in Q$.
By Lemma \ref{Le:Fan90}, we may assume that $y_1$ is the unique neighbor of $w$ in $R$.
Then $y_1$ also is the unique neighbor of $x_1$ in $R$
(as otherwise counting $w,x_1,x_2$ in, we would have $t\geq 3$).
Since $t=2$, the maximum matching between $(R,C)$ has size two,
so by K\"onig's theorem \cite{K31}, either $\{y_1,y_2\}$ or $\{y_1,x_2\}$ is a vertex cover in $(R,C)$.
In the former case, let $z=y_2$ and $H=G[R]$;
and in the latter case, let $z=x_2$ and $H=G[R\cup \{x_2\}]$.
As $G$ is 2-connected and $|V(H)|\geq 3$, $H+y_1z$ is 2-connected;
and every vertex in $H+y_1z$, except $y_1,z$, has the same degree as in $G$, which is at least $k$.
Applying Theorem \ref{Le:FanErdoGal} to $H+y_1z$, there exists a $(y_1,z)$-path in $H+y_1z$ of length at least $k$.
Clearly this path also lies in $H$, which implies an $(x_1,R,x_2)$-path of length at least $k+1$.
By Lemma \ref{Le:clique}(ii) with $t=2$ and $d=k+1$,
$\omega\leq c-k-q\leq c-k-1$, a contradiction.
This proves Claim 3.
\end{proof}

Note that Claim 3 also shows that if there exist infeasible components in $G-C$, then $k\geq 4$.

\begin{claim}\label{claim-not2-con}
For any infeasible component $R$ in $G-C$, $|V(R)|\geq 3$ and $R$ is not 2-connected.
\end{claim}
\begin{proof}
If $|V(R)|\leq 2$, then by Claim 3, any vertex $u\in V(R)$ has degree at most $1+|N_{C}(R)|\leq k-1$ in $G$, a contradiction.
So $|V(R)|\geq 3$.

Suppose for a contradiction that $R$ is 2-connected.
For any $x,y\in V(G)$, let $I_{xy}$ be 1 if $xy\in E(G)$ and 0 otherwise.
Then for any $u\in V(R)$, we have $d_R(u)=d_G(u)-d_C(u)\geq k-t-\sum_{v\in Q}I_{uv}$.
By Theorem \ref{Le:FanErdoGal}, for any two vertices $y,y'\in V(R)$, there is a
$(y,y')$-path of length at least
$$\ell \geq \frac{\sum_{u\in V(R)\backslash\{y,y'\}}(k-t-\sum_{v\in Q}I_{uv})}{|V(R)|-2}
\geq k-t-\frac{q}{|V(R)|-2}.
$$
First we consider that $|V(R)|=3$.
In this case, $R$ is a triangle, say $V(R)=\{y_1,y_2,y_3\}$.
For any $i$, it follows from $|N_R(y_i)|= 2$ that $N_C(y_i)\geq k-2$.
By Claim 3, $N_C(y_i)=N_{C}(R)$ for each $i$ and thus $t=|N_C(R)|=k-2$ and $q=0$.
By Lemma \ref{Le:clique}, $c-k+1\leq \omega\leq c-3(k-3)$.
So $k\leq 4$. Recall that $k\geq 4$.
So we have $k=4$, $t=2$, $q=0$ and $\omega=c-3$.
That is, $ch(R)=(t,q,\omega)=(2,0,c-3)=(2,0,c-k+1)$, a contradiction.

Now we may assume that $|V(R)|\geq 4$. In this case, following the above inequality,
we have $\ell\geq k-t-\frac{q}{2}\geq k-t-\frac{k-2-t}{2}=\frac{k-t}{2}+1$,
where the last inequality holds as $t+q=|N_C(R)|\leq k-2$. By Lemma \ref{Le:clique}(i),
\begin{align*}
c-k+1\leq \omega\leq c-(\ell+1)(t-1)\leq c-\left(\frac{k-t}{2}+2\right)(t-1),
\end{align*}
which implies that $(k-t+4)(t-1)\leq 2(k-1)$,
and thus
\begin{align*}
k(t-3)\leq (t-4)(t-1)-2.
\end{align*}
If $t\geq 4$, then $k\leq \frac{(t-4)(t-1)-2}{t-3}=(t-2)-\frac{4}{t-3}\leq k-2-q-\frac{4}{t-3}\leq k-2$,
a contradiction. If $t=3$, this becomes that $0\leq -4$, which is impossible.
Thus $t=2$. Let $T=\{x_1,x_2\}$ and $x_1y_1,x_2y_2$ be two independent edges for $y_1,y_2\in V(R)$.
Then any $v\in V(R)\backslash \{y_1,y_2\}$ has $N_C(v)\subseteq \{x_1,x_2\}$,
so $d_R(v)\geq k-2$.
Since $R$ is 2-connected, by Theorem \ref{Le:FanErdoGal}, there is a $(y_1,y_2)$-path in $R$
of length at least $k-2$. By Lemma \ref{Le:clique}(ii),
$\omega\leq c-(k-1)-q\leq c-k$ if $q\geq 1$, a contradiction.
So we have $q=0$ and $\omega=c-k+1$. In this case, we have $ch(R)=(t,q,\omega)=(2,0,c-k+1)$.
This proves this claim.
\end{proof}

\begin{claim}\label{Claim-path in B}
Let $R$ be an infeasible component in $G-C$ and
$B$ an end-block of $R$ with the cut-vertex $b$.
Let $T:=\{v\in V(C):|N_{B-b}(v)|\geq 2\}$ with $t:=|T|$.
Then the following hold:\\
(i) $B$ is 2-connected with $|V(B)|\geq 5$;\\
(ii) For any $y\in V(B-b)$, there is a $(y,b)$-path in $B$ of length at least $\frac{2}{3}(k-t+1)$; \\
(iii) For any $y_1, y_2\in V(B-b)$, there is a $(y_1, y_2)$-path in $B$ of length at least $\frac{7}{12}(k-t)$;\\
(iv) $t\leq 2$.
\end{claim}

\begin{proof}
Let $Q:=N_C(B-b)\backslash T$ and $q:=|Q|$.
By Claim 3, we have $t+q\leq k-2$.

(i). For any $v\in V(B-b)$,
$d_R(v)=d_G(v)-d_C(v)\geq k-(k-2)=2$.
So any end-block $B$ of $R$ is 2-connected and thus $|V(B)|\geq 3$.
Suppose that $|V(B)|\in \{3,4\}$.
First we claim that $|N_{C}(B-b)|\geq 2$.
If $|V(B)|=3$, then it is clear, as $k\geq 4$ and every vertex in $B-b$ has degree at most 2 in $B$, there are at least 2 neighbors in $V(C)$.
For $|V(B)|=4$, by the similar argument we also see that $|N_{C}(B-b)|\geq 2$, unless $k=4$ and $B$ is a $K_4$.
In the latter case (say $N_{C}(B-b)=\{x\}$ and $k=4$), since $G$ is 2-connected,
there exists some $x'\in V(C)\backslash \{x\}$ which has a neighbor in $V(R-(B-b))$;
as $B$ is a $K_4$, there exists an $(x,R,x')$-path of length at least 5.
By Lemma \ref{Le:clique} (with $d=5$ and the strong attachment $\{x,x'\}$),
we have $c-3=\omega\leq c-4$, a contradiction. This proves that $|N_{C}(B-b)|\geq 2$.
By Claim 4, there exists another end-block $B_0$ of $R$. Let $b_0$ be the cut-vertex of $R$ with $b_0\in V(B_0)$.
As $|N_{C}(B-b)|\geq 2$, there exist $y\in V(B-b)$ and $y'\in V(B_0-b_0)$ such that
$yx,y'x'\in E(G)$ are independent edges, where $x,x'\in V(C)$.
As $B$ and $B_0$ are 2-connected, there is a $(y,y')$-path of length at least 4.
By Lemma \ref{Le:clique} (with $d=6$ and the strong attachment $\{x,x'\}$), we have
$c-k+1\leq \omega\leq c-5,$ which implies that $k\geq 6$.

Suppose $|V(B)|=3$. Then obviously $B$ is a triangle, say $by_1y_2b$.
And $d_C(y_i)\geq k-2$ for $i=1,2$. On the other hand,
$d_C(y_i)\leq |N_C(R)|\leq k-2$ for $i=1,2$. Thus, $y_1,y_2$ both are adjacent to all vertices in $N_C(B-b)$.
So $T=N_C(B-b)$ and $t=|T|=k-2$. There is a $(y_1,y_2)$-path in $B$ of length 2.
By Lemma \ref{Le:clique} (with $d=4$ and the strong attachment $T$), as $k\geq 6$,
we obtain $\omega\leq c-3(k-3)\leq c-k$, a contradiction.

Suppose $|V(B)|=4$. Then $B$ contains a cycle of length 4, say $by_1y_2y_3b$.
If $|N_C(B-b)|\leq k-3$,
then $d_B(y_i)=3$ for $i=1,2,3$, and this also implies that each of $y_1,y_2,y_3$ is adjacent to all vertices in $N_C(B-b)$.
So $N_C(B-b)$ is a strong attachment of size $k-3$.
Note that $B$ is a $K_4$.
By Lemma \ref{Le:clique} (with $d=5$ and the strong attachment $N_C(B-b)$),
we have $\omega\leq c-4(k-4)\leq c-k$, where the last inequality holds as $k\geq 6$, a contradiction.
So $|N_C(B-b)|\geq k-2$. By Claim 3, $N_C(B-b)=N_C(R)$ is of size $k-2$.
We claim that $N_C(B-b)$ is a strong attachment.
If $Q\neq \emptyset$, choose $x\in Q$.
Suppose that $y_1$ is the unique vertex in $N_{B-b}(x)$.
Then by the degree condition, we see that $y_2$ and $y_3$ are adjacent to every other vertex in $B$
and have the same neighborhood $N_C(B-b)-\{x\}$ in $C$.
So, $N_C(B-b)$ is also a strong attachment.
If $Q=\emptyset$, then $N_C(B-b)=T$ is clearly a strong attachment. This proves the claim.
By Lemma \ref{Le:clique} (with $d=4$ and the strong attachment $N_C(B-b)$),
we have $\omega\leq c-3(k-3)\leq c-k$ (since $k\geq 6$), a contradiction.
This proves (i).

(ii). For $x,y\in V(G)$, let $I_{xy}=1$ if $xy\in E(G)$ and $0$ otherwise.
Then for any vertex $u\in V(B-b)$, we have $d_B(u)\geq k-t-\sum_{v\in Q}I_{uv}$.
Since $B$ is 2-connected, by Theorem \ref{Le:FanErdoGal}, for any $y\in V(B-b)$ there is a $(y,b)$-path of length $\ell_{yb}$,
such that
\begin{align*}
\ell_{yb}&\geq \frac{\sum_{u\in V(B-\{y,b\})}d_B(u)}{|B|-2}\geq k-t-\frac{q}{|B|-2}\geq \frac{2}{3}(k-t+1),
\end{align*}
where the last inequality holds because $|B|\geq 5$ and $q\leq k-t-2$.
This proves (ii).

(iii). Recall that for any $u\in V(B-b)$, $d_B(u)\geq k-t-\sum_{v\in Q}I_{uv}$.
By Theorem \ref{Le:FanErdoGal}, for any distinct $y,y'\in V(B)$, there is a $(y,y')$-path of length $\ell_{yy'}$ at least
$$\frac{\sum_{u\in V(B-\{y,y'\})}d_B(u)}{|B|-2}=\frac{(\sum_{u\in V(B-\{y,y',b\})}d_B(u))+d_B(b)}{|B|-2}\geq$$
$$\frac{(\sum_{u\in V(B-\{y,y',b\})}(k-t-\sum_{v\in Q}I_{vu}))+2}{|B|-2}
\geq \frac{(|B|-3)(k-t)-q+2}{|B|-2}\geq k-t-\frac{k-t+q-2}{|B|-2}.
$$
On the other hand, $|B|^2-|B|\geq 2e(B)\geq \sum_{u\in B-b}(k-t-\sum_{v\in Q}I_{vu})+2=(|B|-1)(k-t)-(q-2)$,
which implies that $|B|\geq k-t-\frac{q-2}{|B|-1}$.
So, $\frac{k-t-2}{|B|-2}\leq 1+\frac{q-2}{(|B|-1)(|B|-2)}$.
Hence,
\begin{align*}
\frac{k-t+q-2}{|B|-2}\leq 1+\frac{q-2}{(|B|-1)(|B|-2)}+\frac{q}{|B|-2}\leq 1+\frac{q-2}{12}+\frac{q}{3}= 1+\frac{5q-2}{12}\leq \frac{5(k-t)}{12},
\end{align*}
since $|B|\geq 5$ and $q\leq k-t-2$. So $l_{yy'}\geq k-t-\frac{k-t+q-2}{|B|-2}\geq \frac{7(k-t)}{12}$. This proves (iii).

(iv). Suppose that $t\geq 3$. Since $T$ is a strong attachment,
by (iii) and Lemma \ref{Le:clique}, we have that
$$c-k+1\leq \omega\leq c-\left(\frac{7(k-t)}{12}+1\right)(t-1),$$
which implies that $(k-t)(7t-19)\leq 0$.
As $t\geq 3$, it follows $k\leq t$, a contradiction to $t\leq k-2-q\leq k-2$. This proves (iv).
\end{proof}

\begin{claim}\label{claim-compoent-2kinds}
There is no infeasible component in $G-C$.
In other words, any component $R$ in $G-C$ has either $|N_C(R)|=k$ or $ch(R)=(2,0,c-k+1)$.
\end{claim}

\begin{proof}
Suppose that there exists an infeasible component $R$ in $G-C$.
By Claim 4, $R$ is not 2-connected,
so there exist two end-blocks $B_1, B_2$ of $R$, with cut-vertices $b_1, b_2$, respectively.
By Claim 5, each $B_i$ is 2-connected and for any vertex $y\in V(B_i-b_i)$,
there exists a $(y,b_i)$-path in $B_i$ of length $\ell_{yb_i}\geq \frac{2(k-t+1)}{3}\geq \frac{2(k-1)}{3}$.

Suppose there exist distinct vertices $x\in N_C(B_1-b_1)$ and $x'\in N_C(B_2-b_2)$.
Then there exist $y\in B_1-b_1$ and $y'\in B_2-b_2$ such that $xy, x'y'$ are two independent edges.
So $\{x,x'\}$ is a strong attachment of $R$ to $C$; and moreover,
there exists an $(x,R,x')$-path of length at least $\ell_{yb_1}+\ell_{y'b_2}+2\geq \frac{4(k-1)}{3}+2$.
By Lemma \ref{Le:clique}, as $k\geq 4$, we have $c-k+1\leq \omega\leq c-\left(\frac{4}{3}(k-1)+1\right)\leq c-k,$
a contradiction.

Therefore, we may assume that $N_C(B_1-b_1)=N_C(B_2-b_2)=\{x\}$ for some vertex $x\in V(C)$.
Let $y$ be a neighbor of $x$ in $B_1-b_1$.
Note that $d_{B_1}(u)\geq k-1$ for any $u\in B_1-b_1$.
By Theorem \ref{Le:FanErdoGal}, there exists a $(y,b_1)$-path of length at least $k-1$.
Since $G$ is 2-connected, there exists an edge $x'y'\in E(G)$ with $x'\in V(C-x)$ and $y'\in V(R)-(B_1-b_1)\cup (B_2-b_2)$.
Clearly $\{x,x'\}$ is a strong attachment of $R$ to $C$ and using the above $(y,b_1)$-path,
one can easily find an $(x,R,x')$-path of length at least $k+1$.
By Lemma \ref{Le:clique}, we have $c-k+1\leq \omega\leq c-k$, a contradiction. This proves Claim \ref{claim-compoent-2kinds}.
\end{proof}

In the remaining, we let $C=x_1x_2\ldots x_cx_1$ and take the index of $x_i$ under modulo $c$.
By Dirac's theorem, $c\geq \min\{n,2k\}$. We also have $c\leq n-1$. This shows that $c\geq 2k$.

\begin{claim}\label{claim-struR}
Let $R$ be a component in $G-C$ with $|N_C(R)|=k$.
Then, there exists $i\in [c]$ such that
$I:=\{x_{i+1},x_{i+3},\ldots, x_{i+2k-3}\}$ is a stable set, $V(C)\backslash I$ is a clique of size $c-k+1$,
and $N_C(R)=\{x_i,x_{i+2},\ldots,x_{i+2k-2}\}$; moreover, $|V(R)|=1$.
\end{claim}

\begin{proof}
Let $N:=N_C(R)$. Let $W$ be a maximum clique in $G[C]$ and $I:=V(C)\backslash W$.
By Claim 1, $N^+$ is stable and thus $|W\cap N^+|\leq 1$.
By the inclusion-exclusion principle, we have $c-k+1\leq |W|=|W\cup N^+|+|W\cap N^+|-|N^+|\leq c+1-k$.
This shows that $|W|=c-k+1, |W\cap N^+|=1$, and $W\cup N^+=V(C)$.
In particular, we have $I\subseteq N^+$. Similarly, one can show that $I\subseteq N^-$.
Thus, $I^+\cup I^-\subseteq N$.
So $k=|N|\geq |I^+\cup I^-|=|I^+|+|I^-|-|I^+\cap I^-|=2(k-1)-|I^+\cap I^-|$,
implying that $|I^+\cap I^-|\geq k-2$.
Since $|I|=k-1$ (and $c\geq 2k$), it is not hard to see that
the indices of the vertices in $I$ must form an arithmetic progression with difference two,
say $I=\{x_{i+1},x_{i+3},\ldots, x_{i+2k-3}\}$ for some $i\in [c]$.
Also since $I^+\cup I^-\subseteq N$ and $|N|=k$, it follows that $N_C(R)=N=\{x_i,x_{i+2},\ldots,x_{i+2k-2}\}$.

For any $x_j, x_{j+2}$ in $N_C(R)$,
since $C$ is locally maximal, there exists some vertex $y\in V(R)$ such that $N_R(x_j)=N_R(x_{j+2})=\{y\}$.
This further implies that $y$ is the unique neighbor in $R$ for every vertex in $N_C(R)$.
If $|V(R)|\geq 2$, then $y$ is a cut-vertex of $G$, contradicting the fact that $G$ is 2-connected.
Thus $|V(R)|=1$. This proves the claim.
\end{proof}

\begin{claim}\label{claim-strucchR}
Let $R$ be a component in $G-C$ with $ch(R)=(2,0,c-k+1)$, and $T:=N_C(R)$.
Then there exists some integer $i\in [c]$ such that $T=\{x_i,x_{i+c-k}\}$ and
$W=G[\{x_i,x_{i+1},\ldots,x_{i+c-k}\}]$ is a clique of size $c-k+1$;
moreover, $G[R\cup T]$ is a clique of size $k+1$ and there are no edges between $V(W)\backslash T$ and $V(C)\backslash V(W)$.
\end{claim}

\begin{proof}
Let $T=\{x,y\}$ and $W$ be a maximum clique in $G[C]$ of size $c-k+1$.

First we show that the longest $(x,y)$-path in $G[C]$ has length at least $c-k$,
with equality if and only if $T=\{x_i,x_{i+c-k}\}$ for some integer $i\in [c]$,
$W=G[\{x_i,x_{i+1},\ldots,x_{i+c-k}\}]$,
and there are no edges between $V(W)\backslash T$ and $V(C)\backslash V(W)$.
We first observe that there are two disjoint subpaths of $C$, say $L_1, L_2$ from $x,y$ to $V(W)$, respectively.
Let the other end of $L_i$ be $a_i$.
Then, as $W$ is a clique, there exists an $(a_1,a_2)$-path in $W$ through all vertices of $W$,
which, together with $L_1$ and $L_2$, gives an $(x,y)$-path in $G[C]$ passing through all vertices of $W$.
Since $|V(W)|=c-k+1$, this $(x,y)$-path has length at least $c-k$.
Now suppose that the longest $(x,y)$-path has length exactly $c-k$.
Let $P_1,P_2$ be the two $(x,y)$-subpaths on $C$.
If $W$ intersects both $P_1-\{x,y\}$ and $P_2-\{x,y\}$,
then 
we could find an $(x,y)$-path through all vertices of $W$ and thus it has length at least $c-k+1$, a contradiction.
So we may assume that $V(W)\subseteq V(P_1)$. This further shows that $V(W)=V(P_1)$.
That is, there exists $i\in [c]$ such that $T=\{x_i,x_{i+c-k}\}$ and $W=G[\{x_i,x_{i+1},\ldots,x_{i+c-k}\}]$.
In this case, if there is some edge $uv$ with $u\in V(W)\backslash T$ and $v\in V(C)\backslash V(W)$,
then one can easily find an $(x_i,x_{i+c-k})$-path of length at least $c-k+1$, a contradiction.

Next we show that the longest $(x,y)$-path in $G[C]$ has length exactly $c-k$ and moreover, $G[R\cup T]$ is a clique of size $k+1$.
To see this, we notice that since $G$ is 2-connected,
$G[R\cup T]+xy$ is 2-connected and every vertex in $G[R\cup T]+xy$, except $x$ and $y$, has degree at least $k$.
By Theorem \ref{Le:FanErdoGal}, the longest $(x,y)$-path $P$ in $G[R\cup T]+xy$ has length at least $k$,
with equality if and only if $G[R\cup T]+xy$ is the union of some cliques $K_{k+1}$'s
which pairwise share the same vertices $x$ and $y$.
For our case, as the deletion of $\{x,y\}$ only results in one component $R$,
the equality holds if and only if $G[R\cup T]+xy$ is a clique $K_{k+1}$.
It is also clear that $P$ lies in $G[R\cup T]$.
Let $P'$ be the longest $(x,y)$-path in $G[C]$, which is of length at least $c-k$.
Then $C':=P\cup P'$ is a cycle of length at least $c$ with the property that $|E(C')\cap E(C, G-C)|=2$.
If $C'$ has length at least $c+1$, it will contradict that $C$ is locally maximal.
So $C'$ must have length $c$, and thus
the longest $(x,y)$-paths in $G[R\cup T]+xy$ and in $G[C]$ are of lengths exactly $k$ and $c-k$, respectively.
This, together with the last paragraph, imply that $\{x,y\}=\{x_i,x_{i+c-k}\}$ for some $i\in [c]$, $W=G[\{x_i,x_{i+1},\ldots,x_{i+c-k}\}]$ is a clique,
and $G[R\cup T]+xy$ is a clique $K_{k+1}$. In particular,
we see $xy\in E(G)$, so $G[R\cup T]$ is a clique $K_{k+1}$.
This proves Claim \ref{claim-strucchR}.
\end{proof}

\begin{claim}\label{claim-tqome}
If there exists a component $R$ in $G-C$ with $ch(R)=(2,0,c-k+1)$, then $G=Z_{n,k,c}$.
\end{claim}

\begin{proof}
Let $R$ be a component in $G-C$ with $ch(R)=(2,0,c-k+1)$.
By Claim \ref{claim-strucchR}, we may assume that $T:=N_C(R)=\{x_1,x_{c-k+1}\}$ and
$W=G[\{x_1,x_2,\ldots,x_{c-k+1}\}]$ is a clique.

Let $A:=V(C)\backslash V(W)$. We first show that for every $x\in A$,
$N_G(x)\subseteq A\cup T$.
Suppose not. In view of Claim \ref{claim-strucchR}, we may assume that
there exists another component $R'$ in $G-C$ which has a neighbor $x$ in $A$
(because $x$ has no neighbors in $R$ or $W\backslash T$).
By Claim \ref{claim-compoent-2kinds},
either $|N_C(R')|=k$ or $ch(R')=(2,0,c-k-1)$.
If $|N_C(R')|=k$, then by Claim \ref{claim-struR},
$N_C(R')$ is a clique with vertices $\{x_i,x_{i+2},\ldots,x_{i+2k-2}\}$ for some $i\in [c]$.
Since $x\in N_C(R')\cap A$ and $A$ only consists of $k-1$ consecutive vertices on $C$,
there must be $y\in N_C(R')\cap (V(W)\backslash T)$.
So $xy\in E(G)$, contradicting Claim \ref{claim-strucchR}.
So assume that $ch(R')=(2,0,c-k-1)$.
Then $N_C(R')=\{x_j,x_{j+(c-k)}\}$ for some $j\in [c]$, where $x_j\in A$.
In this case, we also see that $x_jx_{j+(c-k)}$ is an edge between $V(W)\backslash T$ and $A$,
a contradiction, finishing the proof.

Therefore, as $\delta(G)\geq k$ and $|A\cup T|=k+1$,
we also see that $G[A\cup T]$ induces a $K_{k+1}$.
Together with Claim 8,
this shows that if $R$ is a component in $G-C$ with $ch(R)=(2,0,c-k+1)$,
then $G[C]$ is a union of a clique $K_{k+1}$ and another clique $K_{c-k+1}$ which share the vertices in $N_C(R)$.

Now consider any component $R_0$ in $G-C$ other than $R$.
We just proved $N_C(R_0)\cap A=\emptyset$,
so $N_C(R_0)\subseteq V(W)$.
By Claim \ref{claim-compoent-2kinds},
either $|N_C(R_0)|=k$ or $ch(R_0)=(2,0,c-k-1)$.
Assume that $|N_C(R_0)|=k$.
Let $x_i,x_{i+2},\ldots,x_{i+2k-2}$ be the vertices of $N_C(R_0)$ for some $i$, which are in $V(W)$.
Then there exist two vertices in $(N_C(R_0))^+$, which are also in $V(W)$,
a contradiction to Claim 1 that $(N_C(R_0))^+$ is stable.
So we have $ch(R_0)=(2,0,c-k-1)$.
By the above paragraph, we must have $N_C(R_0)=N_C(R)$;
moreover, by Claim \ref{claim-strucchR}, $G[R_0\cup N_C(R_0)]$ forms a clique $K_{k+1}$.
This shows that $G=Z_{n,k,c}$, proving this claim.
\end{proof}

Hence, by Claims \ref{claim-compoent-2kinds}, \ref{claim-struR} and \ref{claim-tqome},
we may assume that every vertex $y$ in $G-C$ is an isolated vertex with $|N_C(y)|=k$.

\begin{claim}\label{Claim-N(y)}
For any $y,y'\in G-C$, it holds that $N_C(y)=N_C(y')$.
\end{claim}
\begin{proof}
Suppose that $N_C(y)\neq N_C(y')$.
Then there exist distinct indices $i, j\in [c]$
such that $N_C(y)=\{x_i,x_{i+2},\ldots,x_{i+2k-2}\}$ and $N_C(y')=\{x_j,x_{j+2},\ldots,x_{j+2k-2}\}$.
Also $I:=\{x_{i+1},x_{i+3},\ldots,x_{i+2k-3}\}$ and $I':=\{x_{j+1},x_{j+3},\ldots,x_{j+2k-3}\}$ are independent.
Moreover, $W:=V(C)\backslash I$ and $W':=V(C)\backslash I'$ are cliques.
So $|I'\cap W|\leq 1$, implying that $|I'\cap I|\geq k-2$.

First consider the case that $c=2k$. If $k=2$, then without loss of generality,
we may assume that $N_C(y)=\{x_1,x_3\}$, $N_C(y')=\{x_2,x_4\}$, $I=\{x_2\}$, and $W=x_1x_3x_4x_1$ is a triangle.
Then one can easily find a 5-cycle $x_2y'x_4x_3x_1x_2$, contradicting that $C$ is locally maximal.
If $k\geq 3$, then $I'\cap I\neq \emptyset$.
This implies that the indies of the vertices in $I$ and in $I'$ are of the same parity,
so we must have $N_C(y)=N_C(y')$, a contradiction.

Hence we may assume that $c\geq 2k+1$. 
In this case, as $N_C(y)\neq N_C(y')$, we see that $I\neq I'$.
Since $|I'\cap I|\geq k-2$, it holds that $|I'\cap I|=k-2$.
If $k=2$, then $x_{j+1}$ is in the clique $W=V(C)\backslash\{x_{i+1}\}$.
One of $x_{j-1}, x_{j+3}$ cannot be $x_{i+1}$ (by symmetry, say $x_{i+1}\neq x_{j+3}$).
So $x_{j+1}x_{j+3}\in E(G)$. Then $(C-\{x_{j+1},x_{j+2}\})\cup x_jy'\cup y'x_{j+2}\cup x_{j+2}x_{j+1}\cup x_{j+1}x_{j+3}$
is a cycle which is longer than $C$, a contradiction.
Now let $k\geq 3$. Then without loss of generality, we may assume that $j=i+2$.
Since $x_{i+1}, x_{i+2k}\in W'$, we have $x_{i+2k}x_{i+1}\in E(G)$.
Let $P$ be the unique subpath of $C$ from $x_{i+2}$ to $x_{i+2k-2}$ which contains $x_{i+3}$,
and let $P'=x_{i+2k}x_{i+1}\cup x_{i+1}x_i\cup x_iy\cup yx_{i+2}\cup P$ be a path from $x_{i+2k}$ to $x_{i+2k-2}$.
Since $A:=(V(C)\backslash V(P'))\cup \{x_{i+2k},x_{i+2k-2}\}=V(C)-\{x_i,x_{i+1},...,x_{i+2k-3}\}\subseteq W$,
there exists a path from $x_{i+2k}$ to $x_{i+2k-2}$ and consisting of the vertices in $A$,
which, together with $P'$, forms a cycle $C'$ satisfying that $|C'|>|C|$ and $|E(C')\cap E(C,G-C)|\leq 2$.
This contradicts that $C$ is locally maximal, completing the proof of this claim.
\end{proof}

We now prove that $G=W_{n,k,c}$.
By Claim \ref{Claim-N(y)}, we may assume that for all $y\in G-C$,
$N_C(y)=\{x_1,x_3,...,x_{2k-1}\}$.
By Claim \ref{claim-struR}, $I=\{x_2,x_4,...,x_{2k-2}\}$ is an independent set and $W:=V(C)\backslash I$ is a clique of size $c-k+1$.
Therefore, to prove $G=W_{n,k,c}$, it remains to show that for every vertex $x\in I$,
$N_{G}(x)=\{x_1,x_3,...,x_{2k-1}\}$.
Since $\delta(G)\geq k$, it suffices to show that
any vertex $x_i\in I$ cannot be adjacent to some vertex $x$ in $V(C)-\{x_1,x_3,...,x_{2k-1}\}$.
Suppose for a contradiction that $x_ix\in E(G)$, where $i\in \{2,4,...,2k-2\}$.
As $I$ is independent, such $x$ must be in $V(C)-\{x_1,x_2,...,x_{2k-1}\}$.
Let $P, P'$ be two disjoint subpaths in the segment $x_1x_2...x_{2k-1}$ of $C$ from $x_i,x_{i+1}$ to $x_1, x_{2k-1}$, respectively.
Then $Q=xx_i\cup P\cup x_1y\cup yx_{i+1}\cup P'$ is a path from $x$ to $x_{2k-1}$ and passing through some vertex $y\in G-C$.
Note that $A:=(V(C)\backslash V(Q))\cup \{x,x_{2k-1}\}=V(C)-\{x_1,x_2,...,x_{2k-2}\}$ is a subset of the clique $W$.
So there exists a path from $x$ to $x_{2k-1}$ and consisting of all vertices in $A$.
This path, together with $Q$, forms a cycle $C'$, which is longer than $C$ and $|E(C')\cap E(C,G-C)|\leq 2$, a contradiction.
The proof of Lemma \ref{lem:CliqNum2} is completed.
\end{proof}

We now have finished the proof of Theorem \ref{Th:WZ+closure} (and thus Theorem \ref{Th:WZ+closure1}).

\section{Proofs of Theorems \ref{Th:Var1} and \ref{Th:Var2}}

\noindent {\bf Theorem \ref{Th:Var1}.}
{\em
Let $G$ be a 2-connected graph on $n$ vertices with $\delta(G)\geq k$
and let $C$ be a longest cycle in $G$ of length $c\in [10,n-1]$.
If
$
e(G)>\max\left\{f\left(n,k+1,c\right),f\left(n,\left\lfloor\frac{c}{2}\right\rfloor,c\right)\right\},
$
then $\overline{G}=W_{n,k,c}$ or $Z_{n,k,c}$, where $\overline{G}$ denotes the $C$-closure of $G$.
}

\begin{proof}
We derive this from Theorem \ref{Th:main-refined}.
By the discussion in Subsection 2.1, $f\left(n,\left\lfloor\frac{c}{2}\right\rfloor,c\right)\geq f\left(n,\left\lfloor\frac{c}{2}\right\rfloor-1,c\right)$.
So we have $e(G)>\max\left\{f\left(n,k+1,c\right),f\left(n,\left\lfloor\frac{c}{2}\right\rfloor-1,c\right)\right\}$.
By Theorem \ref{Th:main-refined}, either $\overline{G}=W_{n,k,c}$ or $Z_{n,k,c}$,
$G\subseteq W_{n,\lfloor\frac{c}{2}\rfloor,c}$, or $G$ is a subgraph of a member of $\mathcal{X}_{n,c}\cup \mathcal{Y}_{n,c}$ (only when $k=2$ and $c$ is odd).
If $G\subseteq W_{n,\lfloor\frac{c}{2}\rfloor,c}$ or $G$ is a subgraph of a member of $\mathcal{X}_{n,c}\cup \mathcal{Y}_{n,c}$,
then it is easy to see that $e(G)\leq f\left(n,\left\lfloor\frac{c}{2}\right\rfloor,c\right)$,
a contradiction to that $e(G)> f\left(n,\left\lfloor\frac{c}{2}\right\rfloor,c\right)$.
So it must be that $\overline{G}=W_{n,k,c}$ or $Z_{n,k,c}$.
\end{proof}

The proof of Theorem \ref{Th:Var2} is more involved, as we are not guaranteed to be able to use Theorem \ref{Th:main-refined}.
This is because $$\max\left\{f\left(n,k,c\right),f\left(n,\left\lfloor\frac{c}{2}\right\rfloor-1,c\right)\right\}\geq
\max\left\{f\left(n,k+1,c\right),f\left(n,\left\lfloor\frac{c}{2}\right\rfloor-1,c\right)\right\}$$ holds only when
$\lfloor\frac{c}{2}\rfloor-1\geq k+1$.
In fact when $c\leq 2k+3$, this inequality can be reversed.

\medskip

\noindent {\bf Theorem \ref{Th:Var2}.}
{\em
Let $G$ be a 2-connected graph on $n$ vertices with $\delta(G)\geq k$
and let $C$ be a longest cycle in $G$ of length $c\in [10,n-1]$.
If
$
e(G)>\max\left\{f\left(n,k,c\right),f\left(n,\left\lfloor\frac{c}{2}\right\rfloor-1,c\right)\right\},
$
then either $G\subseteq W_{n,\lfloor\frac{c}{2}\rfloor,c}$, or $k=2$, $c$ is odd and $G$ is a subgraph of a member of $\mathcal{X}_{n,c}\cup \mathcal{Y}_{n,c}$.
}

\begin{proof}
Since $e(G)>f(n,\lfloor\frac{c}{2}\rfloor-1,c)=(\lfloor\frac{c}{2}\rfloor-1)(n-c)+h(c+1,\lfloor\frac{c}{2}\rfloor-1)$,
it holds that either $e(G-C)+e(G-C,C)>(\lfloor\frac{c}{2}\rfloor-1)(n-c)$ or $e(G[C])> h(c+1,\lfloor\frac{c}{2}\rfloor-1)$.
If the former case occurs, then by Theorem \ref{Th:BondyStab-2con},
either $G\subseteq W_{n,\lfloor\frac{c}{2}\rfloor,c}$,
or $c$ is odd and $G$ is a subgraph of a member of $\mathcal{X}_{n,c}\cup \mathcal{Y}_{n,c}$ (if this occurs, then $k=2$).
So we may assume that $e(G[C])> h(c+1,\lfloor\frac{c}{2}\rfloor-1)$.
It suffices to show the following
\medskip

\noindent {\bf Claim.}
{\em
Let $G$ be a 2-connected graph on $n$ vertices with $\delta(G)\geq k$ and
$C$ be a locally maximal cycle in $G$ of length $c\in [10,n-1]$.
If
$e(G)>\max\left\{f\left(n,k,c\right),f\left(n,\left\lfloor\frac{c}{2}\right\rfloor-1,c\right)\right\}$
and
$e(G[C])>h(c+1,\lfloor\frac{c}{2}\rfloor-1),$
then $G\subseteq W_{n,\lfloor\frac{c}{2}\rfloor,c}$.
}

\medskip

The remaining proof is similar to the one of Theorem \ref{Th:WZ+closure}.
Let $\overline{G}$ be the $C$-closure of $G$.
By Lemma \ref{Le:Coper}, $C$ remains a locally maximal cycle in $\overline{G}$;
and by Lemma \ref{Le:CoperHC}, $\overline{G}[C]$ is non-Hamiltonian-connected.
Using Lemma \ref{Le:HamConStruc2}, we see that
one of the following holds:
\begin{itemize}
\item[(i)] $\overline{G}[C]$ contains a subset of $\lfloor\frac{c}{2}\rfloor-1$ vertices of degree at most
$\lfloor\frac{c}{2}\rfloor$ in $\overline{G}[C]$, or
\item[(ii)]  $\overline{G}[C]$ contains a subset $S$ of $s-1$ vertices of degree at most
$s$ in $\overline{G}[C]$ for some $2\leq s\leq \lfloor\frac{c}{2}\rfloor-2$ such that
$\overline{G}[C]-S$ is a clique.
\end{itemize}
Suppose that (i) holds. Lemma \ref{Le:c/2-1:c/2} implies $\overline{G}[C]= W_{c,\lfloor\frac{c}{2}\rfloor,c}$.
Following the same augments in Theorem \ref{Th:WZ+closure}, we have $G\subseteq \overline{G}\subseteq W_{n,\lfloor\frac{c}{2}\rfloor,c}$.
Now assume that (ii) holds.
Since $e(\overline{G})\geq e(G)>\max\left\{f\left(n,k,c\right),f\left(n,\left\lfloor\frac{c}{2}\right\rfloor-1,c\right)\right\}$,
by using $k$ instead of $k+1$ in Lemma \ref{lem:CliqNum}, we derive that the clique number of $\overline{G}[C]$ is at least $c-k+2$.
By Lemma \ref{lem:CliqNum2}, we have $\overline{G}\in \{W_{n,k,c}, ~Z_{n,k,c}\}$,
but in each of the two graphs, the corresponding clique number is $c-k+1$, a contradiction.
This proves the claim. Thus we have proved Theorem \ref{Th:Var2}.
\end{proof}

\section{Concluding remarks}
The approach used here seems to be applicable for the following problem of F\"{u}redi, Kostochka and Verstra\"{e}te in \cite{FKV16}:
for $n\geq \frac{3c}{2}$, to describe the structures of 2-connected $n$-vertex graphs with circumference at most $c$, where $c$ is even,
and with at least $f(n,\frac{c}{2}-2,c)$ edges.
We also wonder if a general and clear stability result can hold for $k$-connected graphs $G$ (or even for 3-connected graphs with minimum-degree at least $k$)
for which $G$ has $n$ vertices, circumference $c$ and $e(G)>\max\{f(n,k+a,c),f(n,\lfloor\frac{c}{2}\rfloor-b,c)\}$ for fixed integers $a,b\geq 1$.
Finally we would like to mention that some related problems can be found in \cite{FKL-arxiv}.

\bigskip

\noindent {\bf Acknowledgement.} The first author would like to thank Alexandr V. Kostochka for helpful discussions.

\end{document}